\documentclass[11pt,reqno,a4paper]{amsart}
\usepackage{yhmath,amsmath,amsfonts,amssymb,enumerate,amsthm,appendix,caption,lipsum,xparse}
\usepackage{enumitem}
\usepackage{soul}
\usepackage{mathrsfs}
\usepackage[utf8]{inputenc}
\usepackage{cmap,mathtools}
\usepackage{graphics} 
\usepackage{epsfig} 
\usepackage{graphicx}
\usepackage{epstopdf}
\usepackage[pagewise]{lineno}

\usepackage{parskip}
\usepackage{cmap,mathtools}
\usepackage{mathrsfs}
\usepackage[margin=0.9in]{geometry} 
\usepackage{cite}
\usepackage[utf8]{inputenc}
\usepackage[T1]{fontenc}

\usepackage[foot]{amsaddr}
\usepackage{xcolor}
\usepackage{cancel}
\usepackage{hyperref}
\usepackage[hyperpageref]{backref}
\hypersetup{
    colorlinks=true,
    linkcolor=myblue,
    filecolor=magenta,      
    urlcolor=mygreen,
    citecolor=mygreen,
}
\definecolor{mygreen}{rgb}{0.01,0.6,0.2}
\definecolor{myblue}{rgb}{0.01, 0.18, 1.0}
\newtheorem{theorem}{Theorem}

\newtheorem{lemma}{Lemma}

\theoremstyle{definition}
\newtheorem{definition}[theorem]{Definition}
\newtheorem{remark}[theorem]{Remark}
\newtheorem{example}[theorem]{Example}

\def\dx{\mathrm{d}x}
\numberwithin{equation}{section}
\numberwithin{theorem}{section}
\numberwithin{equation}{section}
\numberwithin{theorem}{section}
\numberwithin{lemma}{section}
\numberwithin{proposition}{section}

\title[Concentration for $(p,N)$-Laplace Equation Without AR Condition]{CONCENTRATION PHENOMENA FOR $(p, N)$-LAPLACE EQUATION UNDER
DISCONTINUOUS NONLINEARITIES AND PENALIZATION METHOD}
\author[Ankit]{Ankit}
\address[Ankit]{Department of Mathematics, Indian Institute of Technology Jodhpur, Rajasthan 342030, India}
\email{p23ma0003@iitj.ac.in}
\author[GM Figueiredo]{ Giovany M. Figueiredo}
\thanks{GMF was partially supported by CNPq and FAPDF}
\address[G.M. Figueiredo]{Universidade de Brasília, Departamento de Matemática, 70910-900 Brasília DF, Brazil}
\email{giovany@unb.br}
\author[A. Sarkar]{Abhishek Sarkar$^*$}
\thanks{AS was supported by DST-INSPIRE Grant DST/INSPIRE/04/2018/002208, *Corresponding author}
\address[A. Sarkar]{Department of Mathematics, Indian Institute of Technology Jodhpur, Rajasthan 342030, India}
\email{abhisheks@iitj.ac.in}
\begin{document}
\begin{abstract}In this paper, we investigate the existence and concentration of solutions to a $(p,N)$-Laplace equation in $\mathbb{R}^N$ involving a discontinuous nonlinearity and critical exponential growth. To establish the existence of solutions, we employ a penalization technique in the sense of Del Pino and Felmer adapted to a locally Lipschitz functional. Furthermore, by combining variational methods with Moser-type iteration techniques, we obtain the concentration behavior of the solutions. Our results contribute to the study of nonlinear elliptic problems with irregular nonlinearities and critical growth phenomena.
\end{abstract}
\maketitle 
\tableofcontents
\section{Introduction}\label{Section1}
Our aim is to study the following problem 
\begin{equation}\label{main problem}
    \left\{
     \begin{aligned}
       -\epsilon^p\Delta_pu-\epsilon^N\Delta_Nu~&+V(x)(|u|^{p-2}u+|u|^{N-2}u)={H}(u-\beta)f(u)~\text{in}~ \mathbb{R}^N,\\
       \int_{\mathbb{R}^N} V(x)|u|^p\dx<&+\infty, \int_{\mathbb{R}^N} V(x)|u|^N\dx<+\infty, \text{ and } \\ u\in W^{1,p}({\mathbb{R}^N})&\cap W^{1,N}({\mathbb{R}^N}),
    \end{aligned}\tag{$\mathscr{P}_{\epsilon,\beta}$}
    \right.
\end{equation}
where $2<p<N$, $\epsilon, \beta>0, \ \Delta_{r}u=\text{div}(|\nabla u|^{r-2}\nabla u)$ for $r\in \{p,N\}$, and ${H}$ is the Heaviside function. More assumptions on $f$ and $V$ will be followed. Many authors in recent decades have focused on Schr\"odinger equation \[\iota \epsilon \frac{\partial \psi}{\partial t}= -\epsilon^2 \Delta \psi + (V(x)+E)\psi -f(\psi)\ \text{ in } \mathbb{R}^N,\] where $\epsilon>0.$ We note the corresponding steady-state problem be described as \begin{equation*}\label{2schro}
    -\epsilon^2 \Delta u + W(x) u=f(u) \text{ in } \mathbb{R}^N
\end{equation*}
which is equivalently written as under the change of variable $x \mapsto \epsilon x$:
\begin{equation*}
    -\Delta u +W(\epsilon x)u= f(u) \text{ in } \mathbb{R}^N.
\end{equation*}
In \cite{Pino-1996}, for $p=2<N$, the authors explored that the solutions of the equation concentrate on the local minimum of $V(x)$ in a bounded domain. Later, as a general case, many authors have also studied the quasilinear problem 
\begin{equation}\label{pschro}
   -\epsilon^p \Delta_p u  + W(x)|u|^{p-2}u = f(u) \text{ in } \mathbb{R}^N.
\end{equation}
In the case when $\epsilon =1$, the problem \eqref{pschro} models several steady state cases for non-Newtonian fluids and pseudo-plastic fluids. When $p=2$, i.e., in the problem \eqref{pschro} when $\epsilon \to 0$, the solutions concentrate at global minimum points of $W$ (see \cite{Wang-1993}).  In \cite{Alves-2005, Alves-2006}, the authors discussed the existence, multiplicity and concentration of positive solutions when $1<p<N$. Moreover, in \cite{Alves-2009}, the authors considered the following equation 
\begin{equation}
    -\epsilon^N \Delta_N u + W(x) |u|^{N-2} u = f(u) \text{ in } \mathbb{R}^N.
\end{equation} Where the source term arises from the Moser-Trudinger type inequality \cite{do-1997}. For continuous nonlinearities with $(p,N)$-Laplace operator, we also refer to \cite{Li-2025}. In addition to other significant works related to the concentration of solutions involving the $(p,q)$-Laplace operator, we also refer to \cite{Zhang-2022}. We also refer to \cite{Chen-2025} for a comprehensive review on Schr\"odinger equation. So far, the literature reviewed above has mainly treated the source term as continuous. 
\par Now, in the direction where the source term could exhibit discontinuous nonlinearity, we recall some notable works. In \cite{Gazzola-2000}, the authors have considered the existence of a solution for the following class of problem
\begin{align*}
    Lu +W(x)u &= f(x,u) \text{ in } \mathbb{R}^N\\ u &>0 \text{ in } \mathbb{R}^N,
\end{align*} where $L$ is a general second-order elliptic operator, $W$ is coercive and continuous, and $f$ is discontinuous function with subcritical growth. Among the other notable works, with discontinuous nonlinearities we refer to \cite{Alves-2013, Alves-2014, Alves-2002, Alves-2011, Badiale-1993, Ambrosio-2023, Vincenzo-2024,Alves-Mukherjee-2021,Kim-2018,Santos-2020, Alves-Yuan-Huang-2021, Zhang-Jia-2021, Liu-Radu-Yuan-2022, Ambrosio-D-2023, Ankit-2025} and the references therein.

Motivated by the works mentioned above, we consider the problem \eqref{main problem} and study the existence and concentration phenomena of solutions corresponding to the information available about $V$. 

The function $V:\mathbb{R}^N\rightarrow\mathbb{R}$ verifies the following conditions:
\begin{enumerate}[label={($\bf V{\arabic*}$)}]
\setcounter{enumi}{0}
\item \label{V1} $V$ is positive continuous function and $V(x)\geq V_{0}>0$, for all $x\in \mathbb{R}^N.$
\item \label{V2} There exists an open bounded set $\Lambda\subset\mathbb{R}^N$ such that
\begin{equation*}
    V_0=\underset{x\in\Lambda}{\text{inf}}~V(x)<\underset{x\in\partial\Lambda}{\text{min}}~V(x).
\end{equation*}
\end{enumerate}
Furthermore, the source-term ${f}$ satisfies following assumptions:
\begin{enumerate}[label={($\bf f{\arabic*}$)}]
\setcounter{enumi}{0}
\item\label{f1} $f$ is continuous and has critical exponential growth, i.e., there exists $\alpha_0>0$ such that
$$ \underset{|t|\rightarrow +\infty}{\lim}|f(t)|\mathrm{exp}({-\alpha|t|^\frac{N}{N-1}})=\begin{cases}
        0~\text{ if}~\alpha>\alpha_0\\
        +\infty~ \text{ if}~ \alpha<\alpha_0.
    \end{cases}$$
    \item\label{f2} There exists $\zeta>0$ such that $F(t)\geq\zeta|t|^{N+1}$ for all $t\in \mathbb{R}$, where $F(s)=\int_{0}^{s}f(t)\mathrm{d}t.$
    \item\label{f3} $\underset{t\rightarrow 0}{\limsup}\frac{f(t)}{t^{N-1}}=0.$
    \item\label{f4} The map $t\mapsto\frac{f(t)}{|t|^{N-2}t}$ is an increasing function for all $ t >0$ and decreasing for all $t <0$.
    \item\label{f5} There exists $\delta>1$ such that $\delta \mathscr{G}(t)\geq \mathscr{G}(st)$ for all $t\in \mathbb{R}$, $s\in[0,1]$ and $\mathscr{G}(s)=sf(s)-NF(s).$
\end{enumerate}
    \begin{remark}\label{rem1.1} Note that, \ref{f2} implies 
     $\underset{|t|\rightarrow \infty}{\text{lim}}\frac{F(t)}{|t|^N}=+\infty$, where $F(s)=\int_{0}^{s}f(t)\mathrm{d}t$. \end{remark}
    \begin{example} Here, we give two examples of the functions that satisfy all the conditions \ref{f1}--\ref{f5}. We define $f:\mathbb{R} \to \mathbb{R}$ as follows:
    \begin{itemize}
    \item[(i)] $f(t)=\mathrm{sgn}(t) |t|^{N-1+a} \exp\left(\alpha_0 |t|^{\frac{N}{N-1}}\right), \ a \in (0,1],\ \alpha_0 >0.$ 
   \item[(ii)] $f(t)=t|t|^{N-2}(\sqrt{t}+\Phi_1(t)),$ where $\Phi_1$ defined in \eqref{def1.5}. 
   \end{itemize}
\end{example}
{\begin{definition}[Weak Solution] We say $u\in\mathbf{X}_{\epsilon}$ (defined in section \ref{section3}) is a weak solution of problem \eqref{main problem}, if there is $\rho_0\in L^{\tilde{\Phi}_1}(\mathbb{R}^N)$(defined in section \ref{section2}) such that
\begin{equation*}
    \sum_{t\in\{p,N\}}\underset{\mathbb{R}^N}{\int}(|\nabla u|^{t-2}\nabla{u}\cdot\nabla v +V(\epsilon x)|u|^{t-2}uv)\dx-\underset{\mathbb{R}^N}{\int}\rho_0v \dx=0,\quad \forall~v\in \mathbf{X}_{\epsilon}
\end{equation*}
 and $\rho_0(x)\in[\underline{f_H}(u(x)), \overline{f_H}(u(x))]$ a.e. in $\mathbb{R}^N$ where $f_H(t)=H(t-\beta)f(t)$ and $$\underline{f_H}(t)=\underset{\delta\rightarrow0}{\lim}~\underset{|s-t|<\delta}{\mathrm{ess}\inf}~f_H(s) \text{ and }\overline{f_H}(t)=\underset{\delta\rightarrow0}{\lim}~\underset{|s-t|<\delta}{\mathrm{ess}\sup}~f_H(s),$$ with \begin{equation}\label{def1.5}
    \Phi_1(t)=\mathrm{exp}({  |t|^\frac{N}{N-1}})-\sum_{j=0}^{N-2} \frac{ |t|^\frac{Nj}{N-1}}{j!},
    \end{equation} and $\tilde{\Phi}_1$ is defined in Definition \ref{complementary}. 
\end{definition}}
Now we state our main theorem.
\begin{theorem}\label{main theorem 1}
   \textbf{(Existence and Concentration)} Assume that conditions \ref{V1}-\ref{V2} and \ref{f1}-\ref{f5} are fulfilled. Then, there exists $\tilde{\epsilon}$,$\tilde{\beta}$ such that for all $\epsilon\in(0,\tilde{\epsilon})$ and $\beta\in(0,\tilde{\beta})$, problem \eqref{main problem} possesses a weak solution $u_{\epsilon,\beta}\in \mathbf{X}_{\epsilon}.$
    Furthermore, if $x_{\epsilon,\beta}\in\mathbb{R}^N$ denotes a point at which $u_{\epsilon,\beta}$ achieves its maximum, then
    $\underset{(\epsilon,\beta)\rightarrow(0,0)}{{\lim}}V(\epsilon x_{\epsilon,\beta})=V_0.$
\end{theorem}
The present work establishes new existence and concentration results for a class of $(p,N)$-Laplace equations with discontinuous nonlinearities and critical exponential growth in $\mathbb{R}^N$. By combining Clarke's nonsmooth critical point theory with variational tools associated with the Moser--Trudinger inequality, we develop a unified analytical framework for treating mixed-growth $(p,N)$ operators under critical exponential conditions. In addition, we demonstrate that the solutions concentrate near the global minima of the potential. To the best of our knowledge, this is the first contribution addressing such a mixed-growth, discontinuous, critically exponential problem in the absence of the Ambrosetti--Rabinowitz condition.

	\medskip
    
{This type of result remains new even when replacing the $(p,N)$-Laplace operator by the pure $N$-Laplacian. In this sense, our theorem extends and complements the study carried out in \cite{Alves-2014}, as we consider a more general operator and employ new arguments that are necessary due to the absence of the Ambrosetti--Rabinowitz condition, the generality of the operator, and the presence of critical exponential growth.}

There has been a growing interest in the analysis of nonlinear partial differential equations featuring discontinuous nonlinearities, due to their significance in several free-boundary
problems in mathematical physics. Key examples of these issues are the obstacle problem, the seepage surface problem, and the Elenbaas equation. We refer to \cite{Chang-1981} and the references cited there for more such applications.

This article is organized as follows. In Section \ref{section2}, we recall some results that are crucial in our proofs. The Section \ref{section3} deals with the existence of solutions for the auxiliary problem \eqref{Auxiliary problem}. Section \ref{section4} is concerned with the study of the autonomous problem related to \eqref{main problem}. Finally, Section \ref{Section5} contains the proof of the main theorem. 

\textbf{Notations:} Throughout the paper, we will use the following notations:
       \begin{itemize}
       \item ${X}\hookrightarrow Y$ denotes the continuous embedding of $X$ into $Y$.
       \item ${X}\hookrightarrow\hookrightarrow Y$ denotes compact embedding of ${X}$ into $Y$.
       \item $o_n(1)$ denotes $o_n(1)\rightarrow0$ as $n\rightarrow+\infty$.
       \item $C_1,C_2,C_3, \cdots$ all are positive constants and may have different values at different places.
       \item $\rightharpoonup$ denotes weak convergence, $\overset{\ast}{\rightharpoonup}$ denotes weak$^{\star}$ convergence and  $\rightarrow $ denotes strong convergence.
       \item $\alpha_N=N\omega_{N-1}^\frac{1}{N-1}$, where $\omega_{N-1}$ is volume of $N-1$ dimensional unit sphere.
       \item $[ u > d]:= \{x \in \mathbb{R}^N: u(x) > d\}.$
       \item $\partial f(x)$ denote Clarke's generalized gradient of function $f$ at $x$.
        \item $\|u\|_{L^{p}}=\left(\int_{\mathbb{R}^N}|u|^p \dx\right)^\frac{1}{p},\ 1\leq p<\infty.$
       \item $B_R \subset \mathbb{R}^N$ denote the ball of radius $R>0$ centered at origin.
       \item $\mathrm{m}(S)$ denotes the $N$-dimensional Lebesgue measure for $S \subset \mathbb{R}^N.$
       \item $X^*$ denotes the dual space of $X.$
       \end{itemize}
\section{Preliminaries}\label{section2}
In this section, we recall some important existing results useful for our arguments.
\begin{definition}[Generalized Directional Derivative]
    The generalized directional derivative of $J$ at $x$ in the direction $h$ is given by
     $$J^{o}(x;h)=\limsup_{y\rightarrow x,\lambda\rightarrow 0}\frac {J (y+\lambda h)-J(y)}{\lambda}.$$ 
     The function $h\mapsto J^{0}(x,h)$ is a subadditive, continuous, and convex function. Hence, its generalized gradient in the Clark sense is given by 
      $$\partial J(x)=\{x^*\in X^*: \langle x^*,h \rangle_X \leq J^{0}(x;h),\ \forall~ h \in X\}.$$
     For each $x\in X$, $\partial J(x)$ is non empty, convex and weak*-compact subset of $X^*.$ Moreover, the function 
     \begin{equation}\label{lambda}
         \lambda(x)=\underset{x^*\in \partial J(x)}{\min}\|x^*\|_{X^*}
     \end{equation} exists and is lower semi-continuous. If $J\in C^{1}(X,\mathbb{R})$, then $\partial J(x)=\{J'(x)\}.$ A point $x_0\in X$ is called a critical point for $J$ if $0\in \partial J(x).$ For more details on this topic, we refer to the monograph \cite[Clark]{Clarke-1983}.
     \end{definition}
      \begin{definition}[Cerami Condition--Non Smooth Version \cite{Stuart-2011}] Let $X$ be a Banach space and  $\phi:X\rightarrow\mathbb{R}$ be a locally Lipschitz functional. Then one say that $\phi$ satisfies the non smooth Cerami condition at level $c\in\mathbb{R}$, denoted as $(C)_c$, if 
      every sequence $\{x_n\}\subseteq X$ such that
      \begin{equation*} \phi(x_n)\rightarrow c \text{ and } (1+\|x_n\|_{X})\lambda(x_n)\rightarrow0,
      \end{equation*}has a strongly convergent subsequence.
       \end{definition}
       
       \begin{definition}[Orlicz space]\label{2.3}
       The Orlicz space associated with $\Theta$ (an $N$-function) as
 \begin{equation*}\label{s6}
       L^{\Theta}(\mathbb{R}^N)=\left\{u\in L^{1}_{loc}(\mathbb{R}^N):\underset{\mathbb{R}^N}{\int}\Theta\left(\frac{|u|}{\lambda}\right)\dx<+\infty ~\text{for some}~\lambda>0\right\}.
      \end{equation*} \end{definition}
\noindent The space $L^{\Theta}(\mathbb{R}^N)$ is Banach  space endowed with following norm
      \begin{equation*}
      \|u\|_{\Theta}=\inf\left\{\lambda >0:\underset{\mathbb{R}^N}{\int}\Theta\left(\frac{|u|}{\lambda}\right)\dx\leq1\right\} .
      \end{equation*} 
      For more on Orlicz spaces, we refer to \cite{Rao-1991}.
      \begin{definition}\label{def2.4}
          We have $E_{\Theta}(\mathbb{R}^N)$ as a subspace of $L^{\Theta}(\mathbb{R}^N)$ in the following way, $$E_{\Theta}(\mathbb{R}^N)=\overline{\left\{u\in L^1_{loc}(\mathbb{R}^N):u \text{ having bounded support on }\mathbb{R}^N\right\}}^{\|\cdot\|_{\Theta}}.$$
          \end{definition}
    Then, equivalently, we also have 
    $$E_{\Theta}(\mathbb{R}^N)=\overline{C_c^{\infty}(\mathbb{R}^N)}^{\|\cdot\|_{\Theta}}.$$
\begin{definition}[Complementary/conjugate of a function]\label{complementary}  The Complementary function $\tilde{\Theta}$ associated with $\Theta$ is given by
         \begin{equation*}\label{s4}
     \tilde{\Theta}(s)=\underset{t\geq 0}{\sup}\{st-\Theta(t)\},\  s\geq0.
     \end{equation*} \end{definition}
     Next, we recall some lemmas without proof, which will be crucial for establishing our arguments.
\begin{lemma}[Moser-Trudinger Inequality {\cite[Lemma 1]{do-1997}}]\label{lemma2.2}
  For any $\alpha>0$, $N\geq 2$ and $u\in W^{1,N}(\mathbb{R}^N)$,
     \begin{equation*}
      \underset{\mathbb{R}^N}{\int}\Phi(\alpha|u|^\frac{N}{N-1})\dx<+\infty.
      \end{equation*}
     Moreover, if $\|\nabla u\|_{N}\leq1$ and $
     \|u\|_{N}\leq M<+\infty$ and $\alpha<\alpha_N$ there exist a constant $C>0$ such that
     \begin{equation*}
      \underset{\|\nabla u\|_N \leq1,\|u\|_N \leq M }\sup \underset{\mathbb{R}^N}{\int} \Phi\left({\alpha |u|^\frac{N}{N-1}}\right)\dx\leq C,
      \end{equation*}
      where $\Phi(t)=\mathrm{exp}({t})-\sum\limits_{j=0}^{N-2} \frac{t^j}{j!}$,~$\alpha_N=N\omega_{N-1}^\frac{1}{N-1} \text{ and } \|u\|_N=(\underset{\mathbb{R}^N}{\int}|u|^N \dx)^\frac{1}{N}.$
\end{lemma}
{Here we would like to note that $\Phi_1(t)=\Phi(t^{\frac{N}{N-1}}).$}
{\begin{lemma}[\cite{Ambrosio-2024}]\label{lemma2.3}
Let $\boldsymbol{Y}$ be a Banach space and $I\in C^1(\boldsymbol{Y},\mathbb{R})$. Let $\{u_n\}_{n\geq1}\in\mathcal{N}_0$ such that $I(u_n)\rightarrow u_0.$ Then $\{u_n\}_{n\geq1} $ has a strongly convergent subsequence in $\boldsymbol{Y}$, where $\mathcal{N}_0$ denote Nehari manifold associated with functional $I,$ defined as 
     \begin{equation*}
      \mathcal{N}_{0}:=\{u\in W^{1,p}_{V_0}({\mathbb{R}^N})\cap W^{1,N}_{V_0}({\mathbb{R}^N})\setminus\{0\}:~\langle I_{V_0}^{'}(u),u\rangle=0\}.
     \end{equation*}
\end{lemma}}
\begin{lemma}[{\cite[Proposition 2.2]{{Alves-2014}}}]\label{lemma2.4}
    Let X be a real Banach space. Suppose that $\varphi\in Lip_{loc}(X,\mathbb{R})$.
Let $\{x_n\}_{n\geq 1}\subset X $ and 
$\{\rho_n\}_{n\geq1}\subset X^*$ with $\rho_n\in\partial \varphi(x_n)$. If $x_n\rightarrow x$ in X and $\rho_n \overset{\ast}{\rightharpoonup}\rho_0$ in $X^*$, then $\rho_0\in\partial \varphi(x).$
\end{lemma}
\begin{lemma}[{\cite[Lemma 2.5]{Ankit-2025}}]\label{lemma 2.5}
Let $\varsigma(t)=\max\{t,t^N \}$ and
 $\tilde{\Phi}_{1}$ be conjugate function associated with $\Phi_{1}$. Then, the following inequalities are satisfied:
\begin{itemize}
\item[(i)]  $\tilde{\Phi}_{1}(\frac{\Phi_{1}(r)}{r})\leq \Phi_{1}(r),~\forall~ r>0$.
\item[(ii)] $\tilde{\Phi}_{1}(tr)\leq \varsigma(t)\tilde{\Phi}_{1}(r), ~\forall~ t,r \geq 0.$ 
\end{itemize}
Hence, $\tilde{\Phi}_{1}\in \Delta_2$ and $E_{\tilde{\Phi}_{1}}(\mathbb{R}^N)=L^{\tilde{\Phi}{1}}(\mathbb{R}^N)$.
\end{lemma}
\begin{lemma}[{\cite[Lemma 2.6]{Ankit-2025}}]\label{2.6}
Let $E_{\Phi_{1}}(\mathbb{R}^N)$(defined in Definition \ref{def2.4}) be a subspace of an Orlicz space $L^{\Phi_{1}}(\mathbb{R}^N)$ . Then, the following embeddings hold:
\begin{itemize}
\item[(i)] $\mathbf{X}_{\epsilon}\hookrightarrow\ E_{\Phi_{1}}(\mathbb{R}^N)$, and
\item[(ii)] $E_{\Phi_{1}}(\mathbb{R}^N)\hookrightarrow L^{N}(\mathbb{R}^N).$
\end{itemize}
\end{lemma}
\section{Auxiliary Problem }\label{section3}
Now, in this section, we will study the auxiliary problem instead of the \eqref{main problem}. The auxiliary problem is found under the change of variable $x \mapsto \epsilon x.$ Our motivation to study under these circumstances is based on the findings of \cite{Pino-1996} and \cite{Alves-2014}. 
\par To begin with, first, we fix $k, a, \beta$~>0 such that $\beta<a<k$ and $\frac{f(a)}{a^{N-1}}=\frac{V_0}{k},$ and we define 
$$ \tilde f(s)=\begin{cases}
        f(s)~,s<a\\
        \frac{V_0s
        ^{N-1}}{k},~s\geq a.
    \end{cases}$$
    Note that $\tilde{f} $ is a continuous function on $\mathbb{R}$. Define
    \begin{equation}\label{defg}
     g(x,t)=\chi_{\Lambda}(x)f(t)+(1-\chi_{\Lambda}(x))\tilde{f}(t).
    \end{equation}
    One can easily observe that the function $g:\mathbb{R}^N\times\mathbb{R}\rightarrow\mathbb{R}$ is Carath\'eodory function such that it satisfy following properties:
    \begin{enumerate}[label={($\bf g{\arabic*}$)}]
\setcounter{enumi}{0}
\item\label{g1} $g(x,t)$ also having critical exponential growth i.e. there exist $\alpha_0>0$ such that
$$ \underset{|t|\rightarrow +\infty}{\lim}|g(x,t)|\mathrm{exp}({-\alpha|t|^\frac{N}{N-1}})=\begin{cases}
        0~\text{ if}~\alpha>\alpha_0\\
        +\infty~ \text{ if}~ \alpha<\alpha_0
    \end{cases}\text{ uniformly} ~\text{a.e.~in}~x\in \mathbb{R}^N.$$
    \item\label{g2} There exists $\zeta>0$ such that $G(x,t)\geq\zeta|t|^{N+1}$, for all $ t\in \mathbb{R}$ and for all $x\in \mathbb{R}^N$, where $G(x,t)=\int_{0}^{t}g(x,s)\mathrm{d}s.$
    \item\label{g3} $\underset{t\rightarrow0}{{\limsup}}\frac{g(x,t)}{t^{N-1}}=0.$
    \item\label{g4} The map $s\mapsto\frac{g(x,t)}{|t|^{N-2}t}$ is an increasing function for all $t>$0 and decreasing for all $ t <0$.
    \item\label{g5} There exists $\delta>1$ such that $\delta \mathscr{G}(x,t)\geq \mathscr{G}(x,st)$, for all $t\in \mathbb{R}$ , $s\in[0,1]$, where $\mathscr{G}(x,t)=tf(x,t)-NG(x,t).$
\end{enumerate}
\begin{remark}\label{remark3.1}
    We note that \ref{g2} implies $\underset{|t|\rightarrow \infty}{{\lim}}\frac{G(x,t)}{|t|^N}=+\infty$ $\text{ uniformly} ~\text{a.e.~in}~x\in \mathbb{R}^N$, where $G(x,t)=\int_{0}^{t}g(x,s)\mathrm{d}s.$
\end{remark}
    So the auxiliary problem is reduced to the following problem under the transformation $x\mapsto\epsilon x$
    \begin{equation}\label{Auxiliary problem}
    \left\{
     \begin{aligned}
       -\Delta_pu-\Delta_Nu~&+V(\epsilon x)(|u|^{p-2}u+|u|^{N-2}u)=H(u-\beta)g(\epsilon x,u)~~ \text{in}~~ \mathbb{R}^N,\\
       \int_{\mathbb{R}^N} V(\epsilon x)|u|^p\dx<&+\infty~, \int_{\mathbb{R}^N} V(\epsilon x)|u|^N\dx<+\infty,\text{ and}\\ u &\in W^{1,p}_{V_{\epsilon}}({\mathbb{R}^N})\cap W^{1,N}_{V_{\epsilon}}({\mathbb{R}^N}),
    \end{aligned}\tag{$\mathscr{P}_{\epsilon,\beta}^{a}$}
    \right.
\end{equation}
where $V_{\epsilon}(x):=V(\epsilon x)$. For the auxiliary problem, our working space will be 
$$\mathbf{X}_{\epsilon}=W^{1,p}_{V_{\epsilon}}({\mathbb{R}^N})\cap W^{1,N}_{V_{\epsilon}}({\mathbb{R}^N})$$
equipped with the norm
$$\|u\|_{\mathbf{X}_{\epsilon}}=\|u\|_{W^{1,p}_{V_{\epsilon}}}+\|u\|_{W^{1,N}_{V_{\epsilon}}},$$ where $$\|u\|_{W^{1,r}_{V_{\epsilon}}}=\bigg(\int_{\mathbb{R}^{N}}|\nabla u|^r~\mathrm{d}x+\int_{\mathbb{R}^{N}}V(\epsilon x)|u|^r~\mathrm{d}x\bigg)^{\frac{1}{r}}~\text{for}~r\in\{p,N\}.$$ 
We say that $u\in\mathbf{X}_{\epsilon}$ is weak solution to \eqref{Auxiliary problem}, if there is $\rho_0\in L^{\Phi_1}(\mathbb{R}^N)$ such that 
\begin{equation*}
    \sum_{t\in\{p,N\}}\underset{\mathbb{R}^N}{\int}(|\nabla u|^{t-2}\nabla{u}\cdot\nabla v +V(\epsilon x)|u|^{t-2}uv)\dx-\underset{\mathbb{R}^N}{\int}\rho_0v\dx=0,\quad \forall~v\in \mathbf{X}_{\epsilon}.
\end{equation*}
 and $\rho_0(x)\in[\underline{g_H}(\epsilon x,u(x)), \overline{g_H}(\epsilon x,u(x))]$ a.e. in $\mathbb{R}^N$ with $g_H(x,t)=H(t-\beta)g( x,t)$, where $$\underline{g_H}(x,t)=\underset{\delta\rightarrow0}{\lim}~\underset{|s-t|<\delta}{\mathrm{ess}\inf}~g_H(x,s) \text{ and }\overline{g_H}(x,t)=\underset{\delta\rightarrow0}{\lim}~\underset{|s-t|<\delta}{\mathrm{ess}\sup}~g_H(x,s). $$
 \begin{remark}\label{remark1}
     If $u\in\mathbf{X}_{\epsilon}$ is weak solution of \eqref{Auxiliary problem} with $u(x)<a$ for all $x\in\Lambda_\epsilon^c$, where $\Lambda_{\epsilon}=\{x\in\mathbb{R}^N:\epsilon x\in\Lambda\}$, then $u$ is solution to \eqref{main problem}.
 \end{remark}
 \subsection{Mountain Pass Geometry of the Energy Functional}
Define the functional $I_{\epsilon,\beta}:\mathbf{X}_{\epsilon}\rightarrow\mathbb{R}$
\begin{equation}\label{eqn3.0}
    I_{\epsilon,\beta}(u)=\frac{1}{p}\underset {\mathbb{R}^N}{\int}(|\nabla u|^p+V(\epsilon x)|u|^p)\dx+\frac{1}{N}\underset {\mathbb{R}^N}{\int}(|\nabla u|^N+V(\epsilon x)|u|^N)\dx-\underset {\mathbb{R}^N}{\int}G_{H}(\epsilon x,u)\dx,
\end{equation}
where $G_{H}( x,t)={\int}_{0}^{t}H(s-\beta)g(x,s)\mathrm{d}s.$\\
Next, we will consider $$I_{\epsilon,\beta}(u)=Q_\epsilon(u)-\Upsilon_{\epsilon,\beta}(u),$$
where $$Q_\epsilon(u)=\frac{1}{p}\underset {\mathbb{R}^N}{\int}(|\nabla u|^p+V(\epsilon x)|u|^p)\dx+\frac{1}{N}\underset {\mathbb{R}^N}{\int}(|\nabla u|^N+V(\epsilon x)|u|^N)\dx$$
and $$\Upsilon_{\epsilon,\beta}(u)= \underset {\mathbb{R}^N}{\int}G_{H}(\epsilon x,u)\dx.$$
\par In the next Lemma, we discuss the mountain pass geometry for the energy functional. 
\begin{lemma}[Mountain Pass Geometry]\label{lemma 3.1}
Assume that \ref{g1}-\ref{g3} hold. Then
\begin{itemize}
\item[(i)] $I_{\epsilon,\beta}(0)=0$ and the functional $I_{\epsilon,\beta}\in Lip_{loc}(\mathbf{X}_{\epsilon}, \mathbb{R}).$
    \item[(ii)] there exist $\varrho>0$ and $r>0$ such that $I_{\epsilon,\beta}(u)\geq r$ for all $\|u\|_{\mathbf{X}_{\epsilon}}=\varrho.$
    \item[(iii)] there is $\vartheta\in\mathbf{X}_{\epsilon}$ such that $I_{\epsilon,\beta}(\vartheta)<0.$
    \end{itemize} \end{lemma}
    
     \begin{proof} (i) $I_{\epsilon,\beta}(0)=0$ and second subpart follows from Lemma $3.1$ of \cite{Ankit-2025}.\\
(ii) From \ref{g1} and \ref{g3}, it follows that for any $\tau>0$ and $\nu>N$, there exist a constant $C_1(\tau)$ (depends only on $\tau$) and $\alpha_0>0$ such that 
\begin{equation*}
    |g(x,t)|\leq \tau|t|^{N-1}+C_1(\tau)|t|^{\nu-1}\Phi_{\alpha_0,N-2}(t),
\end{equation*}
{and therefore, we get
\begin{equation}\label{eqn3.1}
    |G(x,t)|\leq \frac{\tau}{N}|t|^{N}+C_1(\tau)|t|^{\nu}\Phi_{\alpha_0,N-2}(t),
\end{equation}
where \begin{equation*}
    \Phi_{\alpha,N-2}(t)=\mathrm{exp}(\alpha|t|^\frac{N}{N-1})-\sum\limits_{j=0}^{N-2} \frac{(\alpha |t|^\frac{N}{N-1})^j}{j!}.
\end{equation*} }
Using \eqref{eqn3.1}, we have
\begin{equation}\label{eqn3.2}
\underset {\mathbb{R}^N}{\int}G_{H}(\epsilon x,u)\dx\leq\underset {\mathbb{R}^N}{\int}G(\epsilon x,u)\dx\leq\frac{\tau}{N}\underset {\mathbb{R}^N}{\int}|u|^N \dx+C_1(\tau)\underset {\mathbb{R}^N}{\int}|u|^\nu \Phi_{\alpha_0,N-2}(u)\dx.
\end{equation}
Let $ \kappa\in(0,1)$ such that $\|u\|_{\mathbf{X}_{\epsilon}}\leq \kappa $. Choose $\alpha>\alpha_0$ close to $\alpha_0$ such that $\alpha\|u\|_{\mathbf{X_{\epsilon}}}^{\frac{N}{N-1}}<\alpha_N.$  So using \cite[Lemma 2.3]{Marcos-2009} and Sobolev embedding, we obtain
\begin{equation}\label{eqn3.3}
C_1(\tau)\underset {\mathbb{R}^N}{\int}|u|^\nu \Phi_{\alpha_0,N-2}(u)\dx\leq\tilde{C}\|u\|_{\mathbf{X}_{\epsilon}}^{\nu}.
\end{equation}
From \eqref{eqn3.2} and \eqref{eqn3.3}, we get 
\begin{equation}\label{eqn3.4}
\underset {\mathbb{R}^N}{\int}G_{H}(\epsilon x,u)\dx\leq\frac{\tau}{N}\|u\|_{{\mathbf{X}_{\epsilon}}}^N+C_1(\tau)\|u\|_{{\mathbf{X}_{\epsilon}}}^{\nu}.
\end{equation}
Finally, combining \eqref{eqn3.0} and \eqref{eqn3.4}, we have
\begin{align*}
I_{\epsilon,\beta}(u)&\geq\frac{1}{p}\|u\|_{W^{1,p}_{V_{\epsilon}}}^p+\frac{1}{N}\|u\|_{W^{1,N}_{V_{\epsilon}}}^N-\frac{\tau}{N}\|u\|_{{\mathbf{X}_{\epsilon}}}^N-C_1(\tau)\|u\|_{{\mathbf{X}_{\epsilon}}}^{\nu}\\&
\geq \frac{\left(\|u\|_{W_{V_{\epsilon}}^{1,p}}^p+\|u\|_{W_{V_\epsilon}^{1,N}}^N\right)}{N}-\frac{\tau}{N}\|u\|_{{\mathbf{X}_{\epsilon}}}^N-C_1(\tau)\|u\|_{{\mathbf{X}_{\epsilon}}}^{\nu}\\&
\geq\frac{2^{1-N}}{N} \left(\|u\|_{W_{V_\epsilon}^{1,p}}+\|u\|_{W_{V_\epsilon}^{1,N}}\right)^N-\frac{\tau}{N}\|u\|_{{\mathbf{X}_{\epsilon}}}^N-C_1(\tau)\|u\|_{{\mathbf{X}_{\epsilon}}}^{\nu}\\&
=\frac{2^{1-N}}{N}\|u\|_{\mathbf{X}_{\epsilon}}^N-\frac{\tau}{N}\|u\|_{{\mathbf{X}_{\epsilon}}}^N-C_1(\tau)\|u\|_{{\mathbf{X}_{\epsilon}}}^{\nu}\\&
=\left( \frac{2^{1-N}}{N}-\frac{\tau}{N}\right)\|u\|_{\mathbf{X}_{\epsilon}}^N-C_1(\tau)\|u\|_{{\mathbf{X}_{\epsilon}}}^{\nu}.
\end{align*}
Choose $\tau=\frac{1}{2^N}$, then
\begin{align*}
I_{\epsilon,\beta}(u)&\geq\frac{1}{N2^N}\|u\|_{{\mathbf{X}_{\epsilon}}}^N-C_1(\tau)\|u\|_{{\mathbf{X}_{\epsilon}}}^{\nu}
\geq\frac{1}{N2^{N+1}}\|u\|_{{\mathbf{X}_{\epsilon}}}^N-C_1(\tau)\|u\|_{{\mathbf{X}_{\epsilon}}}^{\nu}.
\end{align*}
Using the basic calculus, we obtain that $ \frac{1}{N2^{N+1}}\|u\|_{{\mathbf{X}_{\epsilon}}}^N-C_1(\tau)\|u\|_{{\mathbf{X}_{\epsilon}}}^{\nu}$ attain its maxima at some point say $\varrho\in(0,\kappa].$ So for all $u\in\mathbf{X}_{\epsilon}$ such that $\|u\|_{\mathbf{X}_{\epsilon}}=\varrho$, we have
\begin{equation*}
    I_{\epsilon,\beta}(u)\geq r, ~\forall~\|u\|_{\mathbf{X}_{\epsilon}}=\varrho.
\end{equation*}
(iii) Define $\Lambda_{\epsilon}=\{x\in\mathbb{R}^N:\epsilon x\in\Lambda\}$. Choose $\eta\in C_{c}^{\infty}(\mathbb{R}^N)\setminus\{0\}$ such that $\text{supp}(\eta)\subset\Lambda_{\epsilon}.$ For $t>0$, we set $u=t\eta$, and then we obtain
\begin{equation*}
   I_{\epsilon,\beta}(t\eta) =\frac{t^p}{p}\|\eta\|_{W^{1,p}_{V_{\epsilon}}}^p+\frac{t^N}{N}\|\eta\|_{W^{1,N}_{V_{\epsilon}}}^N-\underset {\mathbb{R}^N\cap~\mathrm{supp}(\eta)}{\int}G_{H}(\epsilon x,t\eta)\dx.
\end{equation*} 
Due to \ref{g2} (Remark \ref{remark3.1}), for any $\tilde{M}>0$ there is $C_ {\tilde{M}}>0$ such that
\begin{equation*}
    G(x,t)\geq \tilde{M}|t|^N-C_{\tilde{M}},~\forall~(x,t)\in\mathbb{R}^N\times \mathbb{R}.
\end{equation*}
Therefore, we get the following estimate
\begin{equation*}
  I_{\epsilon,\beta}(t\eta)\leq \frac{t^p}{p}\|\eta\|_{W^{1,p}_{V_{\epsilon}}}^p+\frac{t^N}{N}\|\eta\|_{W^{1,N}_{V_{\epsilon}}}^N-\tilde{M}t^N\underset{\mathbb{R}^N\cap\text{supp}(\eta)}{\int}\eta^N \dx+C_{\tilde{M}}m(\text{supp}(\eta)\cap\mathbb{R}).
\end{equation*}
Choose $\tilde{M}>0 $ such that $ \frac{1}{N}\|\eta\|_{W^{1,N}_{V_{\epsilon}}}^N-\tilde{M}\underset{\mathbb{R}^N\cap\text{supp}(\eta)}{\int}\eta^Ndx<0$. Now for large enough $t$, we have $I_{\epsilon,\beta}(t\eta) \rightarrow -\infty$ and hence the proof of (iii) follows.
\end{proof}
\subsection{Compactness of the Energy Functional--Cerami Condition} {In this subsection, we discuss the compactness of the energy functional $I_{\epsilon,\beta}$. Before we begin, we first recall a lemma that ensures the existence of a Cerami sequence.
\begin{lemma} [{\cite[Theorem 3.1]{Livrea-2009}}] \label{lemma 2.1}
Let X be a real Banach space and $\phi:X\rightarrow\mathbb{R}$ be a locally Lipschitz functional. Suppose we have 
$$\max\{\phi(0),\phi(x_1)\}<\underset{\|x\|_{X}\leq \rho}{\inf}\phi(x)$$
for some $x_1\in X$ with $\|x_1\|_{X}>\rho.$ Let 
$$c=\underset{\gamma\in \Gamma}{\inf}\underset{t\in [0,1]}{\max}J(\gamma(t)),$$
where $\Gamma$ is the collection of paths joining $0$ and $x_1$. Then there is a non-smooth Cerami sequence for the functional $\phi.$ 
\end{lemma}}
By the virtue of Lemmas \ref{lemma 3.1} and \ref{lemma 2.1}, we obtain a Cerami sequence $ \{u_n\}\subset{\mathbf{X}_{\epsilon}}$ such
\begin{equation}\label{CC} 
I_{\epsilon,\beta}(u_n)\rightarrow c_{\epsilon,\beta}~  \text{and}~ (1+\|u_n\|_{\mathbf{X}_{\epsilon}}) \lambda(u_n)\rightarrow 0,
      \end{equation}  
      where $\lambda$ defined in \eqref{lambda} and
      \begin{equation*}
         c_{\epsilon,\beta}=\underset{\gamma\in \Gamma_{\epsilon,\beta}}{\inf}\underset{t\in [0,1]}{\max}I_{\epsilon,\beta}(\gamma(t)), 
      \end{equation*}
      and
      \begin{equation*}
        \Gamma_{\epsilon,\beta}=\{\gamma\in (C[0,1],\mathbf{X}_{\epsilon}):I_{\epsilon,\beta}(0)=0, I_{\epsilon,\beta}(\gamma(1))<0\}. 
      \end{equation*} The next lemma establishes that such a sequence is bounded. We are influenced by the methods discussed in \cite[Lemma 2.2]{Fang-2009} and \cite[Lemma 3.2]{Lam-2013} to prove the preceding lemma.  
\begin{lemma}\label{lemma 3.2}
   Let \ref{g1}, \ref{g2} and \ref{g5} hold, then Cerami sequence $\{u_n\}$ is bounded in $\mathbf{X}_{\epsilon}$. \end{lemma}
    \begin{proof}
      On contrary suppose that the sequence $\{u_n\}$ is not bounded in $\mathbf{X}_{\epsilon}.$ This means 
      \begin{equation*}
          \|u_n\|_{\mathbf{X}_{\epsilon}}\rightarrow+\infty~\text{as}~~n\rightarrow+\infty.
      \end{equation*}
      Define $v_n=\frac{u_n}{\|u_n\|_{\mathbf{X}_{\epsilon}}}.$ Then $v_n$ is bounded in $\mathbf{X}_{\epsilon} $ as $\|u_n\|_{\mathbf{X}_{\epsilon}}=1$, for all $n\in \mathbb{N}.$ So, up to a subsequence still denoted by itself, we can assume that 
      \begin{equation*}
     \begin{cases}
     v_n\rightharpoonup v~~\text{ in~~}\mathbf{X}_{\epsilon},\\
     v_n(x)\rightarrow v(x)~\text{a.e. ~in~}~\mathbb{R}^N.
     \end{cases}
     \end{equation*}
     {\bf Claim:} $v=0$ a.e. in $\mathbb{R}^N.$\\
     Consider the set  $S=\{x\in\mathbb{R}^N:v(x)\neq0\}$ such that $m(S) \neq 0$. If $x\in S$, then,  we have
     $|u_n(x)|=|v_n(x)|\|u_n\|_{\mathbf{X}_{\epsilon}}\rightarrow+\infty$, for $x\in S$ as $n\rightarrow+\infty.$ From assumption \ref{g2} (Remark \ref{remark3.1}), for each $x\in S,$ we get
     \begin{equation}\label{limit}
         \underset{n\rightarrow+\infty}{\lim}\frac{G(x,u_n(x))}{|u_n(x)|^N}\frac{|u_n(x)|^N}{\|u_n\|_{\mathbf{X}_{\epsilon}}^N}=\underset{n\rightarrow+\infty}{\lim}\frac{G(x,u_n(x))}{|u_n(x)|^N}|v_n(x)|^N=+\infty~~     
     \end{equation}
     and $G(x,t)\geq K_1$, for all $(x,t)\in S\times \mathbb{R}$, for some constant $K_1$. So 
     \begin{equation*}
         \frac{G(x,u_n)-K_1}{\|u_n\|_{\mathbf{X}_{\epsilon}}^N}\geq0,~\forall~x\in S~\text{and}~\forall~n\in\mathbb{N}.
     \end{equation*}
     That is
     \begin{equation*}
         \frac{G(x,u_n)|v_n(x)|^N}{|u_n(x)|^N}-\frac{K_1}{\|u_n\|_{\mathbf{X}_{\epsilon}}^N}\geq0,~\forall~x\in S~\text{and}~\forall~n\in\mathbb{N}.
     \end{equation*}
     From \eqref{CC}, it follows that as $n \to \infty$, we obtain
     \begin{align*}
         c_{\epsilon,\beta}&=I_{\epsilon,\beta}(u_n)+o_n(1), \\& 
         \geq\frac{1}{N}\left(\|u_n\|_{W^{1,p}_{V_{\epsilon}}}^p+\|u_n\|_{W^{1,N}_{V_{\epsilon}}}^N \right)-\underset{\mathbb{R}^N}{\int}G_{H}(\epsilon x,u_n)\dx+o_n(1).
         \end{align*}
         Hence
         \begin{equation}\label{eq1}
         \underset{\mathbb{R}^N}{\int}G_{H}(\epsilon x,u_n)\dx\geq\frac{1}{N}\left(\|u_n\|_{W^{1,p}_{V_{\epsilon}}}^p+\|u_n\|_{W^{1,N}_{V_{\epsilon}}}^N \right)-c_{\epsilon,\beta}+o_n(1) ~\text{ as } n \rightarrow+\infty.
         \end{equation}
         Similarly,
         \begin{equation}\label{eq2}
         \underset{\mathbb{R}^N}{\int}G_{H}(\epsilon x,u_n)\dx\leq\frac{1}{p}\left(\|u_n\|_{W^{1,p}_{V_{\epsilon}}}^p+\|u_n\|_{W^{1,N}_{V_{\epsilon}}}^N \right)-c_{\epsilon,\beta}+o_n(1), ~\text{as n}\rightarrow+\infty.
         \end{equation}
         We have $G_H(\epsilon x,u_n)\leq G(\epsilon x,u_n)$, for large $n$,
         and then by \eqref{limit} and Fatou's Lemma, 
         \begin{equation}\label{eq3}
             +\infty=\underset{S}{\int}\underset{n\rightarrow+\infty}{{\liminf}}\frac{G(\epsilon x,u_n)}{|u_n|^N}|v_n|^N \dx\leq\underset{n\rightarrow+\infty}{{\liminf}}\underset{S}{\int}\frac{G(\epsilon x,u_n)}{\|u_n\|_{\mathbf{X}_{\epsilon}}^N}\dx.
         \end{equation}
         Since $$ \|u_n\|_{\mathbf{X}_{\epsilon}}=\|u_n\|_{W^{1,p}_{V_{\epsilon}}}+\|u_n\|_{W^{1,N}_{V_{\epsilon}}}\rightarrow+\infty,~\text{as}~n\rightarrow+\infty.$$
         Hence, three cases arise:
         \begin{itemize}
             \item[(i)] Both $\|u_n\|_{W^{1,p}_{V_{\epsilon}}} $ and $ \|u_n\|_{W^{1,N}_{V_{\epsilon}}}$ are unbounded.
             \item[(ii)] $\|u_n\|_{W^{1,p}_{V_{\epsilon}}} $ is bounded and $ \|u_n\|_{W^{1,N}_{V_{\epsilon}}}$ are unbounded.
             \item[(iii)]
             $\|u_n\|_{W^{1,p}_{V_{\epsilon}}} $ is unbounded and $ \|u_n\|_{W^{1,N}_{V_{\epsilon}}}$ are bounded.
         \end{itemize}
         Suppose case(i) holds. Then there exists $N_0\in\mathbb{N}$ such that $\|u_n\|_{W^{1,p}_{V_{\epsilon}}}>1 $ and $\|u_n\|_{W^{1,N}_{V_{\epsilon}}}>1$ for all $n\geq N_0.$ This implies $\|u_n\|_{W^{1,p}_{V_{\epsilon}}} ^p
         \leq\|u_n\|_{W^{1,N}_{V_{\epsilon}}}^N$.
         As $n \to \infty$, from \eqref{eq1}, we have
         \begin{align*}
         c_{\epsilon,\beta}&\geq\frac{1}{N}\left(\|u_n\|_{W^{1,p}_{V_{\epsilon}}}^p+\|u_n\|_{W^{1,N}_{V_{\epsilon}}}^N \right)-\underset{\mathbb{R}^N}{\int}G_{H}(\epsilon x,u_n)\dx+o_n(1)\\
             & \geq\frac{1}{N2^{p-1}}\|u_n\|_{\mathbf{X}_{\epsilon}}^p-\underset{\mathbb{R}^N}{\int}G_{H}(\epsilon x,u_n)\dx+o_n(1).
         \end{align*}
         From  \eqref{eq3}, and $\|u_n\|_{W^{1,p}_{V_{\epsilon}}} ^p
         \leq\|u_n\|_{W^{1,N}_{V_{\epsilon}}}^N$ for all $n\geq N_0$, we have 
         \begin{align*}
            +\infty \leq\underset{n\rightarrow+\infty}{{\liminf}}\underset{S}{\int}\frac{G(\epsilon x,u_n)}{\|u_n\|_{\mathbf{X}_{\epsilon}}^N}\dx
            &\leq\underset{n\rightarrow+\infty}{{\liminf}}\underset{S}{\int}\frac{G(\epsilon x,u_n)}{\|u_n\|_{\mathbf{X}_{\epsilon}}^p}\dx\\&
            =\underset{n\rightarrow+\infty}{{\liminf}}\frac{\underset{S}{\int }G(\epsilon x,u_n)\dx}{N2^{p-1}(c_{\epsilon,\beta}+\underset{S}{\int}G(\epsilon x,u_n)\dx-o_n(1))}.
         \end{align*}
         As $n\rightarrow +\infty $, we have $$+\infty\leq \frac{1}{N2^{p-1}}.$$ which is not possible. So $v=0~\text{a.e. in  }~\mathbb{R}^N.$\\
Now, consider the case (ii) holds. This means there exists $M>0$ such that $\|u_n\|_{W^{1,p}_{V_{\epsilon}}}\leq M.$ Then, as $n \to \infty$, we get 
\begin{align}\label{eq2.3.9}
         c_{\epsilon,\beta}&\leq\frac{1}{p}\left(\|u_n\|_{W^{1,p}_{V_{\epsilon}}}^p+\|u_n\|_{W^{1,N}_{V_{\epsilon}}}^N \right)-\underset{\mathbb{R}^N}{\int}G_{H}(\epsilon x,u_n)\dx+o_n(1),\notag \\
&\leq\frac{M^p}{p}+\frac{1}{p}\|u_n\|_{W^{1,N}_{V_{\epsilon}}}^N-\underset{\mathbb{R}^N}{\int}G_{H}(\epsilon x,u_n)\dx+o_n(1).\end{align} 
Therefore, as $n \to \infty$, we derive
         \begin{align*} 
         \frac{1}{p}\|u_n\|_{\mathbf{X}_{\epsilon}}^N &\geq c_{\epsilon,\beta}-\frac{M^p}{p}+\underset{\mathbb{R}^N}{\int}G_{H}(\epsilon x,u_n)\dx+o_n(1). \end{align*} As $n \to \infty$, we have
         \begin{align*}
         \|u_n\|_{\mathbf{X}_{\epsilon}}^N\geq pc_{\epsilon,\beta}-{M^p}+p\underset{\mathbb{R}^N}{\int}G_{H}(\epsilon x,u_n)\dx+o_n(1).
         \end{align*}
         Again, using a similar argument for case (i), we have $$+\infty\le\frac{1}{p},$$ which is absurd. So $v=0~\text{a.e. in  }~\mathbb{R}^N.$ 
         One can repeat this argument for the case (iii) also. Therefore, we have proven that $v=0$ almost everywhere in $\mathbb{R}^N$ in all cases.\\ 
         {Next, we define a continuous function $K_n:[0,1]\rightarrow\mathbb{R}$ such that $K_n(t)=I_{\epsilon,\beta}(tu_n)$. Since $I_{\epsilon,\beta}$ is locally Lipschitz is also continuous, $K_n$ is also continuous on the compact interval $[0,1]$.}
       For each $n\in\mathbb{N}$, there is a $t$ (depends on $n$), say $t_n$ such that 
         \begin{equation}\label{eq4}
             I_{\epsilon,\beta}(t_nu_n)=\underset{t\in[0,1]}{\text{max}} ~I_{\epsilon,\beta}(tu_n)
         \end{equation}
         Note that, $t_n\in(0,1]$ for all $n$. Indeed, if $t_n=0$ for some $n\in \mathbb{N}$, as $I_{\epsilon,\beta}(u_n)\rightarrow c_{\epsilon,\beta}$, so there is $n_0\in \mathbb{N}$ such that $ I_{\epsilon,\beta}(u_n)\geq\frac{c_{\epsilon,\beta}}{2}$ for all $n\geq n_0,$ and there holds
         \begin{equation*}
             0<\frac{c_{\epsilon,\beta}}{2}<I_{\epsilon,\beta}(u_n)\leq\underset{t\in[0,1]}{\text{max}} ~I_{\epsilon,\beta}(tu_n)=\underset{t\in[0,1]}{\text{max}} ~I_{\epsilon,\beta}(t_nu_n)=I_{\epsilon,\beta}(0)=0,
         \end{equation*}
         which is a contradiction. So $t_n\in(0,1]$ for all $n$. Let $\{r_k\}$ be a sequence of positive real numbers such that $r_k>1$ and $\underset{k\rightarrow+\infty}{\text{lim}}{\text{r}_k}=+\infty.$ Then $\|r_kv_n\|_{\mathbf{X}_{\epsilon}}=r_k$ for any $k$ and $n$.
         For each $k\in \mathbb{N}$ and due to \ref{g1} and \ref{g2} (Remark \ref{remark3.1}), we have 
         \begin{equation*}
         |G_{H}(\epsilon x,r_kv_n)|\leq\tau r_k^N|v_n|^N+C_1(\tau)|r_k|^\nu|v_n|^\nu\Phi_{\alpha,N-2}(r_k^{\frac{N}{N-1}}v_n^{\frac{N}{N-1}}).
         \end{equation*}
         For fixed $k\in\mathbb{N}$, the following hold a.e. in $\mathbb{R^N}$, as $k \to \infty$
         \begin{equation*}
             r_kv_n(x)\rightarrow0~~\text{ and }  G_H(\epsilon x,r_kv_n)\rightarrow 0.
         \end{equation*}
So, by the Lebesgue dominated convergence theorem, we deduce
         \begin{equation} \label{eqstr}
         \underset{n\rightarrow+\infty}{\text{lim}}\underset{\mathbb{R}^N}{\int}G_{H}(\epsilon x,r_kv_n)\dx=0,~\text{for each fixed } k\in\mathbb{N}.
          \end{equation}
         For large $n$, we note $\frac{r_k}{\|u_n\|_{\mathbf{X}_{\epsilon}}}\in(0,1)$ because $\|u_n\|_{\boldsymbol{X}_{\epsilon}}\rightarrow+\infty$ as $n\rightarrow+\infty.$
         By the definition in \eqref{eq4}, we have 
         
         \begin{align}\label{eq2star}
        I_{\epsilon,\beta}(t_nu_n)&\geq I_{\epsilon,\beta}\left(\frac{r_ku_n}{\|u_n\|_{\mathbf{X}_{\epsilon}}}\right) 
            =I_{\epsilon,\beta}(r_kv_n) \notag\\&
            =\frac{1}{p}\|r_kv_n\|_{W^{1,p}_{V_{\epsilon}}}^p+\frac{1}{N}\|r_kv_n\|_{W^{1,N}_{V_{\epsilon}}}^N -\underset{\mathbb{R}^N}{\int}G_{H}(\epsilon x,r_kv_n)\dx \notag\\&
            \geq\frac{1}{N}\left(\frac{N}{p}\|r_kv_n\|_{W^{1,p}_{V_{\epsilon}}}^p+\|r_kv_n\|_{W^{1,N}_{V_{\epsilon}}}^N\right) -\underset{\mathbb{R}^N}{\int}G_{H}(\epsilon x,r_kv_n)\dx \notag\\&
            \geq\frac{1}{N} r_k^p\left(\|v_n\|_{W^{1,p}_{V_{\epsilon}}}^N+\|r_kv_n\|_{W^{1,N}_{V_{\epsilon}}}^N\right)^N-\underset{\mathbb{R}^N}{\int}G_{H}(\epsilon x,r_kv_n)\dx \notag\\&
            \geq\frac{1}{N2^{N-1}}r_k^p-\underset{\mathbb{R}^N}{\int}G_{H}(\epsilon x,r_kv_n)\dx.
         \end{align}
         Now by using \eqref{eqstr} and \eqref{eq2star}, we deduce
         \begin{equation}\label{eq5}
             \underset{n\rightarrow+\infty}{\limsup}~I_{\epsilon,\beta}(t_nu_n)=+\infty.
         \end{equation}
         For $\omega_n\in\partial I_{\epsilon,\beta}(t_nu_n)$, we obtain
         \begin{align*}
          I_{\epsilon,\beta}(t_nu_n) &=I_{\epsilon,\beta}(t_nu_n)-\frac{1}{N}\langle\omega_n,t_nu_n\rangle+o_n(1)\\
          &=\frac{t_n^p}{p}\|u_n\|_{W^{1,p}_{V_{\epsilon}}}^p+\frac{t_n^N}{N}\|u_n\|_{W^{1,N}_{V_{\epsilon}}}^N-\underset {\mathbb{R}^N}{\int}G_{H}(\epsilon x,t_nu_n)\dx\\ &\quad-\frac{t_n^p}{N}\|u_n\|_{W^{1,p}_{V_{\epsilon}}}^p-\frac{t_n^N}{N}\|u_n\|_{W^{1,N}_{V_{\epsilon}}}^N+\frac{1}{N}\underset {\mathbb{R}^N}{\int}\rho _n(t_nu_n)\dx+o_n(1)
         \end{align*}
         where $$\rho_n(x)\in[\underline{g_H}(\epsilon x,u_n(x)),\overline{g_H}(\epsilon x,u_n(x))]=\begin{cases}
             0,~u_n(x)<\beta\\
             [0,g(\epsilon x,\beta)], ~~u_n(x)=\beta\\
             \{g(\epsilon x,u_n(x))\},~u_n(x)>\beta 
             \end{cases},
         \text{ a.e.} ~\text{in}~\mathbb{R}^N. $$
          Now thanks to \ref{g5} and for large enough $n$, we get
         \begin{align*}
          I_{\epsilon,\beta}(t_nu_n)&=\frac{t_n^p}{p}\|u_n\|_{W^{1,p}_{V_{\epsilon}}}^p+\frac{t_n^N}{N}\|u_n\|_{W^{1,N}_{V_{\epsilon}}}^N-\underset {\mathbb{R}^N}{\int}G(\epsilon x,t_nu_n)\dx-\frac{t_n^p}{N}\|u_n\|_{W^{1,p}_{V_{\epsilon}}}^p-\frac{t_n^N}{N}\|u_n\|_{W^{1,N}_{V_{\epsilon}}}^N\\&\quad +\frac{1}{N}\underset {\mathbb{R}^N}{\int}g(\epsilon x,u_n(x)t_n)(t_nu_n)\dx+o_n(1)\\&
          =\frac{t_n^p}{p}\|u_n\|_{W^{1,p}_{V_{\epsilon}}}^p+\frac{t_n^N}{N}\|u_n\|_{W^{1,N}_{V_{\epsilon}}}^N-\frac{1}{N}\underset {\mathbb{R}^N}{\int}[NG(\epsilon x,t_nu_n)-g(\epsilon x, u_n(x)t_n)t_nu_n]\dx\\&\quad-\frac{t_n^p}{N}\|u_n\|_{W^{1,p}_{V_{\epsilon}}}^p-\frac{t_n^N}{N}\|u_n\|_{W^{1,N}_{V_{\epsilon}}}^N+o_n(1)\\&
          \leq \frac{t_n^p}{p}\|u_n\|_{W^{1,p}_{V_{\epsilon}}}^p+\frac{t_n^N}{N}\|u_n\|_{W^{1,N}_{V_{\epsilon}}}^N-{\delta}\underset {\mathbb{R}^N}{\int}G_{H}(\epsilon x,u_n)\dx-\frac{t_n^p}{N}\|u_n\|_{W^{1,p}_{V_{\epsilon}}}^p-\frac{t_n^N}{N}\|u_n\|_{W^{1,N}_{V_{\epsilon}}}^N\\ &\quad+\frac{\delta}{N}\underset {\mathbb{R}^N}{\int}g(\epsilon x,u_n(x))(u_n)dx+o_n(1)\\
          &\leq I_{\epsilon,\beta}(u_n)-\left(\frac{t_n^p}{N}\|u_n\|_{W^{1,p}_{V_{\epsilon}}}^p+\frac{t_n^N}{N}\|u_n\|_{W^{1,N}_{V_{\epsilon}}}^N-\frac{\delta}{N}\underset {\mathbb{R}^N}{\int}g(\epsilon x,u_n(x))(u_n)\dx\right)+o_n(1).
          \end{align*}
          Due to condition $(1+\|u_n\|_{\mathbf{X}_{\epsilon}}) \lambda(u_n)\rightarrow 0 ~\text{as}~n\rightarrow+\infty $, the sequence $$\bigg\{\frac{t_n^p}{N}\|u_n\|_{W^{1,p}_{V_{\epsilon}}}^p+\frac{t_n^N}{N}\|u_n\|_{W^{1,N}_{V_{\epsilon}}}^N-\frac{\delta}{N}\underset {\mathbb{R}^N}{\int}g(\epsilon x,u_n(x))u_n\dx\bigg\} $$ is bounded. So as $n\rightarrow+\infty,$ we have
          $$\underset{n \rightarrow+\infty}{\text{lim}} I_{\epsilon,\beta}(t_nu_n)\leq C_2$$ for some constant $C_2>0$, which contradicts to \eqref{eq5}. Hence sequence  $\{u_n\}_{n\geq1}$ is bounded in $\mathbf{X}_{\epsilon}.$ 
          \end{proof}
\begin{lemma}\label{lemma 3.3}
      Assume that \ref{g1}-\ref{g3} and \ref{g5} hold. Let $\{u_n\}$ be a Cerami sequence such that $\underset{n\rightarrow+\infty}{{\limsup}}\|u_n\|_{W^{1,N}}^{N'}<\frac{\alpha_N}{\alpha_0}.$ Then, there exists $u_{\epsilon,\beta} \in\mathbf{X}_{\epsilon}$ such that $\nabla u_n\rightarrow \nabla u_{\epsilon,\beta}$ \text{a.e.} in $\mathbb{R}^N.$ \end{lemma}
      \begin{proof}
         Using Lemma \ref{lemma 3.2}, up to a subsequence still denoted by itself, we can assume that 
\begin{equation*}
     \begin{cases}
     u_n\rightharpoonup u_{\epsilon,\beta}~\text{in~~}~\mathbf{X}_{\epsilon}\\
     u_n(x)\rightarrow u_{\epsilon,\beta}(x)~\text{a.e. ~in~}~\mathbb{R}^N\\
     u_n\rightarrow u_{\epsilon,\beta} ~\text{in}~L^q(B_R), ~\forall~q\in[1,\infty) ~\text{and for some R>0.}
     \end{cases}
     \end{equation*}
     Due to the boundedness of the sequence $\{ u_n\}$, we obtain$(1+\|u_n\|_{\mathbf{X}_{\epsilon}}) \lambda(u_n)\rightarrow0, ~\text{as}~n\rightarrow+\infty $ which is equivalent to assume $\lambda(u_n)\rightarrow0$ as $n\rightarrow+\infty$. Fix $R>0$. Let $\psi\in C_c ^{\infty}(\mathbb{R}^N)$, $0\leq\psi\leq1$ and
     \begin{equation*}
         \psi\equiv1~~\text{in}~B_{R} \text{ and } \psi\equiv0~~\text{in}~B_{2R}^c
     \end{equation*}
     and $\|\nabla\psi\|_{\infty}\leq ~C_3$, where $C_3$ does not depend on $R$. Now for each $u_n$, there is a $\omega_n\in\partial I_{\epsilon,\beta}(u_n)$ such that as $n\rightarrow +\infty$, we get
     \begin{equation}\label{3.6}
     \begin{split}
         \langle\omega_n, (u_n-u_{\epsilon,\beta})\psi\rangle=o_n(1)\\
         \langle Q'_{\epsilon,\beta}(u_n)-\rho_n,(u_n-u_{\epsilon,\beta})\psi\rangle=o_n(1)
         \end{split}
     \end{equation} 
     where $\rho_n(x)\in [\underline{g_H}(\epsilon x,u_n(x)),\overline{g_H}(\epsilon x,u_n(x))]$ a.e. in $\mathbb{R}^N.$
     Due to Simon's inequality \cite[Lemma 4.2]{Bahrouni-2018}, we get
     \begin{align}\label{3.7}
     C_N^{-1}\underset{{B_{2R}}}{\int}|\nabla u_n-\nabla u_{\epsilon,\beta}|^N \dx &\leq C_N^{-1}\left(\underset{B_{2R}}{\int}|\nabla u_n-\nabla u_{\epsilon,\beta}|^Ndx +V(\epsilon x)|u_n-u_{\epsilon,\beta}|^N \dx\right) \notag\\
         &\leq \sum_{t\in\{p,N\}}\underset{B_{2R}}{\int}\bigg[\left(|\nabla u_n|^{t-2}\nabla u_n-|\nabla u_{\epsilon,\beta}|^{t-2}\nabla u_{\epsilon,\beta}\right)\cdot\left(\nabla u_n -\nabla u_{\epsilon,\beta}\right) \notag\\ &\quad+V(\epsilon x)\left(|u_n|^{t-2}u_n-|u_{\epsilon,\beta}|^{t-2}u_{\epsilon,\beta}\right)\left(u_n-u_{\epsilon,\beta}\right)\bigg]\dx.
     \end{align}
     From equations \eqref{3.6} and \eqref{3.7}, we have 
     \begin{align}\label{3.8}
       C_N^{-1}\underset{{B_{2R}}}{\int}|\nabla u_n-\nabla u_{\epsilon,\beta}|^N \dx &\leq - \sum_{t\in\{p,N\}}  \big[\underset{\mathbb{R}^N}{\int} \left(|\nabla u_n|^{t-2}\nabla u_n-|\nabla u_{\epsilon,\beta}|^{t-2}\nabla u_{\epsilon,\beta}\right)(u_n-u_{\epsilon,\beta})\nabla\psi \dx \notag\\ &\quad+ \underset{\mathbb{R}^N}{\int}\rho_n\psi(u_n-u_{\epsilon,\beta} )\dx \big] +  o_n(1).
     \end{align}
     Now using the H\"older's inequality, compact embedding $W^{1,N}(B_{2R})\hookrightarrow \hookrightarrow L^p(B_{2R})$ for all $p\in[1,+\infty)$, and boundedness of $u_n$, for $t\in\{p,N\}$, we have
     \begin{align}\label{eq3.9}
      \left|  \underset{\mathbb{R}^N}{\int}\left(|\nabla u_n|^{t-2}\nabla u_n-|\nabla u_{\epsilon,\beta}|^{t-2}\nabla u_{\epsilon,\beta}\right)(u_n-u_{\epsilon,\beta})\nabla\psi \dx\right|
       &\leq\|\nabla\psi\|_{\infty}\left( \|\nabla u_n\|_{t}^{t-1}+\|\nabla u_{\epsilon,\beta}\|_{t}^{t-1}\right) \notag \\
       &\quad \times \left(\underset{B_{2R}}{\int}|u_n-u_{\epsilon,\beta}|^{t'}\dx\right)^{t'}.\end{align} We observe that the RHS of \eqref{eq3.9} tends to $0$ as $n$ tends to $\infty.$
Now, we also have \begin{equation}\label{eqn3.100} 
\left| \underset{\mathbb{R}^N}{\int}\rho_n\psi(u_n-u_{\epsilon,\beta} )\dx\right|\leq \underset{{B_{2R}}}{\int}|\rho_n||u_n-u_{\epsilon,\beta}|\dx.\end{equation}
     From \ref{g1} and \ref{g3}, it follows that for any $\tau>0$ and $\nu>N$ there exist a constant $C_1(\tau)$ (depends on $\tau$) and $\alpha_0$ such that 
\begin{equation}\label{eqn3.101}
    |g(x,t)|\leq \tau|t|^{N-1}+C_1(\tau)|t|^{\nu-1}\Phi_{\alpha_0,N-2}(t).
\end{equation}
Then, from \eqref{eqn3.100} and \eqref{eqn3.101}, we get
\begin{equation}\label{eq3.10}
 \underset{{B_{2R}}}{\int}|\rho_n||u_n-u_{\epsilon,\beta}|\dx\leq\underset{{B_{2R}}}{\int} \left(\tau|u_n|^{N-1}+C_1(\tau)|u_n|^{\nu-1}\Phi_{\alpha_0,N-2}(u_n)\right)|u_n-u_{\epsilon,\beta}|\dx =:I_1+I_2
\end{equation}
where $$I_1=\underset{{B_{2R}}}{\int} \left(\tau|u_n|^{N-1}\right)|u_n-u_{\epsilon,\beta}|\dx \text{ and } I_2=\underset{{B_{2R}}}{\int} C_1(\tau)|u_n|^{\nu-1}\Phi_{\alpha_0,N-2}(u_n)|u_n-u_{\epsilon,\beta}|\dx.$$
\\Using H\"older's inequality one can easily show that $I_1\rightarrow0$ as $n\rightarrow+\infty.$
Now for $I_2$, Choose $t>1$ and $\alpha>\alpha_0$ such that $\alpha t\|u_n\|_{W^{1,N}}^{N'}<\alpha_N$ for all $n\geq N_0$, and $\frac{1}{t}+\frac{1}{t'}=1$. This is possible only due to the assumption $\underset{n\rightarrow+\infty}{{\limsup}}\|u_n\|_{W^{1,N}}^{N'}<\frac{\alpha_N}{\alpha_0}.$ Next, we estimate $I_2$ and we obtain
\begin{align*}
    I_2 &=\underset{{B_{2R}}}{\int} C_1(\tau)|u_n|^{\nu-1}\Phi_{\alpha,N-2}(u_n)|u_n-u_{\epsilon,\beta}|\dx\\
    &\leq C_1(\tau)\left( \underset{{B_{2R}}}{\int}|u_n|^{(\nu-1)t'}|u_n-u_{\epsilon,\beta}|^{t'}\dx\right)^\frac{1}{t'}\left(\underset{{B_{2R}}}{\int} \left(\Phi_{\alpha,N-2}(u_n)\right)^t\dx\right)^\frac{1}{t}\\
    &\leq C_1(\tau)\left( \underset{{B_{2R}}}{\int}|u_n|^{(\nu-1)t'}|u_n-u_{\epsilon,\beta}|^{t'}\dx\right)^\frac{1}{t'}\left(\underset{{B_{2R}}}{\int} \left(\Phi_{\alpha t\|u_n\|_{W^{1,N}}^{N'},N-2}\left(\frac{u_n}{\|u_n\|_{W^{1,N}}}\right)\right)\dx\right)^\frac{1}{t}\\
    &\leq C_3\left( \underset{{B_{2R}}}{\int}|u_n|^{(\nu-1)t'}|u_n-u_{\epsilon,\beta}|^{t'}\dx\right)^\frac{1}{t'}.
\end{align*}
Choose $\nu=t'N+1 $ and apply again H\"older's inequality, we obtain
\begin{align}\label{eq3.11}
I_2 &\leq C_3\left( \underset{{B_{2R}}}{\int}|u_n|^{t'^2N}|u_n-u_{\epsilon,\beta}|^{t'}\dx\right)^\frac{1}{t'} \notag \\ & \leq C_3\left(\left(\underset{{B_{2R}}}{\int}|u_n|^\frac{t'^2N^2}{N-1}\dx \right)^\frac{N-1}{N}\left(\underset{{B_{2R}}}{\int}|u_n-u_{\epsilon,\beta}|^{Nt'} \dx\right)^{\frac{1}{N}}\right)^\frac{1}{t'} \notag\\
        &\leq C_4\left(\left(\underset{{B_{2R}}}{\int}|u_n-u_{\epsilon,\beta}|^{Nt'} \dx\right)^{\frac{1}{N}}\right)^\frac{1}{t'}\rightarrow0 ~\text{as} ~n\rightarrow+\infty.
\end{align}
From equations \eqref{3.7}, \eqref{eq3.9}--\eqref{eq3.11}, we have
\begin{equation*}
    \underset{n\rightarrow+\infty}{{\lim}}\underset{B_{2R}}{\int}|\nabla u_n-\nabla u_{\epsilon,\beta}|^N\dx=0.
\end{equation*}
Since $R>0$ was arbitrary, we get $\nabla u_n(x)\rightarrow \nabla u_{\epsilon,\beta}(x)$ a.e. in $\mathbb{R}^N.$
      \end{proof}
\begin{lemma}\label{lemma3.4}
       Assume that \ref{V1} and \ref{g2} hold. 
       Then mountain pass level $c_{\epsilon,\beta}$ satisfies
       \begin{equation}\label{eq3.12}
           0<c_{\epsilon,\beta}<c_0=:\frac{N-p}{pN}\left(\frac{\alpha_N}{\alpha_0}\left(\min\{1,V_0\}\right)^\frac{N^2}{N-1}-M^\frac{N}{N-1}\right)^\frac{p(N-1)}{N}
       \end{equation}
where, $M={\max}\{M_1,M_2\}$  such that $\|u_n\|_{\mathbf{X}_{\epsilon}} \leq M_1$ and we choose $M_2$ such that $\frac{\alpha_N}{\alpha_0}\left(\min\{1,V_0\}\right)^N<M_2$. Moreover the sequence $\{u_n\}$ satisfies
       \begin{equation}\label{eqn3.21}
           \underset{n\rightarrow+\infty}{{\limsup}}\|u_n\|_{W^{1,N}}^{N'}<\frac{\alpha_N}{\alpha_0}. 
       \end{equation} \end{lemma}
       \begin{proof}
           Let $\xi\in C_{c}^\infty(\mathbb{R}^N)$ such that $0\leq\xi\leq1$,
           $$\xi\equiv 1~\text{in}~B_{\frac{1}{2}} \text{ and }
               \xi\equiv 0~\text{in}~B_{1}^c,$$
           and $m(\{t\xi>\beta\})>0$ for $t>0.$ For such $\xi$, we get
           \begin{equation*}
             I_{\epsilon,\beta}(t\xi) =\frac{t^p}{p}\|\xi\|_{W^{1,p}_{V_{\epsilon}}}^p+\frac{t^N}{N}\|\xi\|_{W^{1,N}_{V_{\epsilon}}}^N-\underset {\{t\xi>\beta\}}{\int}G_{H}(\epsilon x,t\xi)\dx  .
           \end{equation*}
           From assumption \ref{g2}, we have $\zeta>0$ such that
           \begin{equation*}
           I_{\epsilon,\beta}(t\xi)\leq\frac{t^p}{p}\|\xi\|_{W^{1,p}_{V_{\epsilon}}}^p+\frac{t^N}{N}\|\xi\|_{W^{1,N}_{V_{\epsilon}}}^N-\zeta t^{N+1}\underset {\mathbb{R}^N}{\int} \ |\xi|^{N+1} \dx  
           \end{equation*}
       There is $\zeta_1$ such that $t\xi\in\Gamma$ for all $\zeta>\zeta_1$ and using the elementary calculus, we obtain 
        \begin{equation*}
            0\leq c_{\epsilon,\beta}\leq\underset{t\in[0,1]}{\text{max}} ~I_{\epsilon,\beta}(t\xi)=A\left(\frac{A}{B(N+1)\zeta}\right)^\frac{p}{N-p+1}\left(\frac{1}{p}-\frac{1}{N+1}\right).
        \end{equation*}
        where $A=\|\xi\|_{W^{1,p}_{V_{\epsilon}}}^p+\|\xi\|_{W^{1,N}_{V_{\epsilon}}}^N$ and $B=\underset {\mathbb{R}^{N}}{\int}|\xi|^{N+1}dx.$ Hence $A\left(\frac{A}{B(N+1)\zeta}\right)^\frac{p}{N-p+1}\left(\frac{1}{p}-\frac{1}{N+1}\right) \rightarrow0 $ as $\zeta\rightarrow+\infty$. Due to this, there is a $\zeta_0>0$ such that for all $\zeta>\zeta_0>\zeta_1$, \eqref{eq3.12} holds. As $n\rightarrow+\infty$, we have
        $$c_{\epsilon,\beta}+o_n(1)\geq\left(\frac{1}{p}-\frac{1}{N}\right)\|u_n\|_{W_{V_{\epsilon}}}^p$$
        $$\underset{n\rightarrow+\infty}{{\limsup}}\|u_n\|_{{W^{1,p}_{V_{\epsilon}}}}\leq\left(\frac{c_{\epsilon,\beta}Np}{N-p}\right)^\frac{1}{p}.$$
        Since the sequence $\{u_n\}$ is bounded in $\mathbf{X}_{\epsilon},$ we deduce
        $$ \|u_n\|_{{W^{1,p}_{V_{\epsilon}}}}+\|u_n\|_{{W^{1,N}_{V_{\epsilon}}}}\leq M \text{  and  } \underset{n\rightarrow+\infty}{{\limsup}}\|u_n\|_{{W^{1,N}_{V_{\epsilon}}}}\leq M-\left(\frac{c_{\epsilon,\beta}Np}{N-p}\right)^\frac{1}{p}.$$
       Due to assumption \ref{V1}, {by embedding $W_{V_{\epsilon}}^{1,N} (\mathbb{R}^N)\hookrightarrow W^{1,N}(\mathbb{R}^N)$}, it follows that
        \begin{equation*}
            \begin{split}
              \underset{n\rightarrow+\infty}{\limsup} \|u_n\|_{W^{1,N}}^\frac{N}{N-1}\leq\left(\left(\text{min}\{1,V_0\}\right)^{-N} \underset{n\rightarrow+\infty}{{\limsup}}\|u_n\|_{{W^{1,N}_{V_{\epsilon}}}}\right)^\frac{N}{N-1}\\
              =\left(\left(\text{min}\{1,V_0\}\right)^\frac{-N^2}{N-1}\left(M-\left(\frac{c_{\epsilon,\beta}Np}{N-p}\right)^\frac{1}{p} \right) \right)^\frac{N}{N-1}\\
              \leq \left(\text{min}\{1,V_0\}\right)^\frac{-N^2}{N-1}\left( M^\frac{N}{N-1}+\left( \frac{c_{\epsilon,\beta}Np}{N-p}\right)^\frac{N}{p(N-1)}\right). 
            \end{split}
        \end{equation*}
        Hence, we obtain \eqref{eq3.12} and conclude the proof.
       \end{proof}
       \begin{lemma}\label{tight}
         Let $\{u_n\}$ be a Cerami sequence for $I_{\epsilon,\beta}.$ Then for all $\xi>0$, there exist $R=R(\xi)>0$ such that 
         \begin{equation}
             \underset{n\rightarrow+\infty}{{\limsup}} \sum_{t\in\{p,N\}}\left(\underset{B_R^c}{\int} |\nabla u_n|^{t}+V(\epsilon x)|u_n|^{t}\right)\dx<\xi.
         \end{equation}
       \end{lemma}
       \begin{proof}
           For $R>0$, consider $\tilde\psi_R\in C_c^{\infty}(\mathbb{R}^N)$ such that $0\leq\tilde\psi_R\leq1$ with $$\tilde\psi_R\equiv0~\text{in}~B_{\frac{R}{2}}~\text{and} ~\tilde\psi_R\equiv1~\text{in}~B_{R}^c,$$
           and $|\nabla\tilde\psi_R|\leq\frac{C}{R}$, where $C>0$ is a constant independent of $R>0.$ Next, choose $R>0$ such that $\Lambda_\epsilon\subset B_{\frac{R}{2}}.$ The boundedness of $\{u_n\}_{n\geq1}$ in $\boldsymbol{X_\epsilon}$ implies that $\{u_n\tilde\psi_R\}_{n\geq1}$ is also bounded in $\boldsymbol{X_\epsilon}$. As $n\rightarrow+\infty$, from \ref{CC}, it follows that 
           \begin{equation*}
           \begin{split}
               \langle w_n, u_n\tilde\psi_R\rangle=o_n(1),\\
               \langle Q'(u_n)-\rho_n,u_n\tilde\psi_R \rangle=o_n(1),\\
               \langle Q'(u_n),u_n\tilde\psi_R\rangle-\langle \rho_n,u_n\tilde\psi_R\rangle=o_n(1).
               \end{split}
           \end{equation*}
           where $w_n\in\partial I_{\epsilon,\beta}(u_n)=Q'(u_n)-\rho_n.$ Further, we deduce
           \begin{equation*}
              \sum_{t\in\{p,N\}}\underset{\mathbb{R}^N}{\int}|\nabla u_n|^{t-2}\nabla u_n\cdot \nabla (\tilde{\psi}_Ru_n)+V(\epsilon x)|u_n|^{t-2}u_n(\tilde \psi_R u_n)\dx=o_n(1)+\underset{\mathbb{R}^N}{\int}\rho_nu_n\tilde{\psi}_R\dx,\end{equation*}
              and this implies 
              \begin{equation*}
              \begin{split}
             \sum_{t\in\{p,N\}}\underset{\mathbb{R}^N}{\int} |\nabla u_n|^{t-2}(\nabla u_n\cdot\nabla u_n)\tilde\psi_R +|\nabla u_n|^{t-2}u_n\nabla u_n\cdot\nabla\tilde\psi_R+V(\epsilon x)|u_n|^{t-2}u_n(\tilde\psi_R u_n)\dx\\=o_n(1)+\underset{\mathbb{R}^N}{\int}\rho_nu_n\tilde{\psi}_R\dx.
             \end{split}
             \end{equation*}
             Further, we have
             \begin{align*}
             \sum_{t\in\{p,N\}}\underset{\mathbb{R}^N}{\int} |\nabla u_n|^{t-2}(\nabla u_n\cdot\nabla u_n)\tilde\psi_R \dx &+\sum_{t\in\{p,N\}}\underset{\mathbb{R}^N}{\int}|\nabla u_n|^{t-2}u_n\nabla u_n\cdot\nabla\tilde\psi_R\\ &+V(\epsilon x)|u_n|^{t-2}u_n(\tilde\psi_R u_n)\dx=o_n(1)+\underset{\mathbb{R}^N}{\int}\rho_nu_n\tilde{\psi}_R\dx.
           \end{align*}
           Since 
           \begin{equation*}
           \underset{\mathbb{R}^N}{\int}\rho_nu_n\tilde{\psi}_R\dx\leq\underset{\mathbb{R}^N}{\int}\overline{g_H}(\epsilon x,u_n)u_n\tilde{\psi}_R\dx\leq\underset{\mathbb{R}^N}{\int}g(\epsilon x,u_n)u_n\tilde{\psi}_R\dx,
           \end{equation*}
           we obtain
           \begin{align*}
               \sum_{t\in\{p,N\}}\underset{\mathbb{R}^N}{\int} |\nabla u_n|^{t-2}(\nabla u_n\cdot\nabla u_n)\tilde\psi_R &+V(\epsilon x)|u_n|^{t-2}u_n(\tilde\psi_R u_n)\dx \\
               &\leq\sum_{t\in\{p,N\}}\underset{\mathbb{R}^N}{\int}|\nabla u_n|^{t-2}|u_n||\nabla u_n||\nabla\tilde\psi_R|\dx\\ &\quad+\underset{\mathbb{R}^N}{\int}g(\epsilon x,u_n)u_n\tilde{\psi}_R\dx+o_n(1)\\
               &\leq\frac{C}{R} \sum_{t\in\{p,N\}}\underset{\mathbb{R}^N}{\int}|\nabla u_n|^{t-1}|u_n|\dx+\underset{\mathbb{R}^N}{\int}g(\epsilon x,u_n)u_n\tilde{\psi}_R\dx\\&\quad+o_n(1).
               \end{align*}
               Using the fact that $\Lambda_\epsilon\subset B_{\frac{R}{2}}$ with definition in  \ref{defg} and boundedness of $u_n\in\boldsymbol{X_{\epsilon}}$ with H\"older's inequality, there exist a constant $C_1>0$ such that we have 
               \begin{equation*}
                  \left(1- \frac{1}{k}\right)\sum_{t\in\{p,N\}}\left(\underset{\mathbb{R}^N}{\int} |\nabla u_n|^{t}+V(\epsilon x)|u_n|^{t}\right)\tilde\psi_R\dx
               \leq\frac{C_1}{R},
               \end{equation*}
               So choose $R>0$ large enough such that 
               \begin{equation*}
                  \underset{n\rightarrow+\infty}{{\limsup}} \sum_{t\in\{p,N\}}\left(\underset{B_R^c}{\int} |\nabla u_n|^{t}+V(\epsilon x)|u_n|^{t}\right)\dx
               \leq\frac{C_1}{R}<\xi.
               \end{equation*} This completes the proof.
       \end{proof}
       \subsection{Existence of Critical Point of the Energy Functional Corresponding to the Auxiliary Problem}We begin this subsection by proving the existence of nontrivial critical points for the energy functional corresponding to the problem \eqref{Auxiliary problem}.
       

     \begin{lemma}\label{lemma3.5}
      Assume that \ref{V1}, \ref{g1}-\ref{g3}, and \ref{g5} hold. Then the sequence $\{u_n\}$ satisfies non-smooth Cerami $c$-condition for $0< c<c_0$ and $u_{\epsilon,\beta}$ is a non trivial weak solution for problem \eqref{Auxiliary problem}. Here $c_0$ is the same number defined in Lemma \ref{lemma3.4}. \end{lemma}
       \begin{proof}
          By the virtue of Lemma \ref{lemma 3.3} and 
 \ref{lemma3.4}, it follows that $\nabla u_n\rightarrow\nabla u_{\epsilon,\beta}$ a.e. in $\mathbb{R}^N$. For $t\in\{p,N\},$ the sequence $\{|\nabla u_n|^{t-2}\nabla u_n\}$  is bounded in $L^{\frac{t}{t-1}}(\mathbb{R}^N)$. So,
\begin{align*}
 \begin{split}
|\nabla u_n|^{t-2}\nabla u_n\rightharpoonup |\nabla u_{\epsilon,\beta}|^{t-2}\nabla u_{\epsilon,\beta} \text{ in } L^{\frac{t}{t-1}}(\mathbb{R}^N),
|\nabla u_n|^{t-2}\nabla u_n\rightarrow |\nabla u_{\epsilon,\beta}|^{t-2}\nabla u_{\epsilon,\beta} \text{ a.e. in } \mathbb{R}^N.
  \end{split}
 \end{align*}

Using density arguments for $t\in\{p,N\}$, one can prove that 
\begin{equation}\label{3.15}
\underset{n\rightarrow+\infty}{\lim}\underset{\mathbb{R}^N}{\int}|\nabla u_n|^{t-2}\nabla u_n \cdot\nabla v \dx = \underset{\mathbb{R}^N}{\int}|\nabla u_\epsilon|^{t-2}\nabla u_{\epsilon,\beta}\cdot\nabla v \dx, ~\forall v\in \mathbf{X}_{\epsilon},
\end{equation} and
\begin{equation}\label{3.16}
    \underset{n\rightarrow+\infty}{\lim}\underset{\mathbb{R}^N}{\int}V(x)| u_n|^{t-2} u_n  v \dx = \underset{\mathbb{R}^N}{\int}V(x)| u_{\epsilon,\beta}|^{t-2} u_{\epsilon,\beta} v \dx, ~~\forall v\in \mathbf{X}_{\epsilon}. 
    \end{equation}
{\bf Claim:} For each $\rho_n\in\partial_tG(\epsilon x,u_n(x))$, there is $\rho_0\in L^{\tilde{\Phi}_1}(\mathbb{R}^N)$ such that
\begin{equation} \label{eqn3.24}
\underset{\mathbb{R}^ N}{\int}\rho_n v \dx\rightarrow \underset{\mathbb{R}^ N}{\int}\rho_0 v \dx,~\forall v\in \mathbf{X}_{\epsilon}.
\end{equation}
{Note that $\Phi_1$ is an increasing and convex function and there is $\alpha>\alpha_0$ such that  
\begin{equation*}
\underset{\mathbb{R}^N}{\int}\tilde{\Phi}_{1}(\rho_n)\dx\leq \underset{\mathbb{R}^N}{\int}\tilde{\Phi}_{1}\left(\tau_1|u_n|^{N-1}+C_1(\tau)|u_n|^\nu\Phi_{N-2,\alpha}(u)\right)\dx,
\end{equation*}
\begin{equation*}
\underset{\mathbb{R}^N}{\int}\tilde{\Phi}(\rho_n)\dx\leq \frac{1}{2}\underset{\mathbb{R}^N}{\int}\varsigma(2\tau)\tilde{\Phi}_{1}(|u_n|^{N-1})\\dx+ \frac{\varsigma(2C_1(\tau))}{2}\underset{\mathbb{R}^N}{\int}\tilde{\Phi}_{1}\left(|u_n|^\nu\Phi_{1}\left(\alpha^\frac{N-1}{N}|u_n|\right)\right)\dx,
\end{equation*}}
By Lemma \ref{lemma 2.5}, there exist constant $C_1>0$ and $C_2>0$ such that 
\begin{equation*}
\underset{\mathbb{R}^N}{\int}\tilde{\Phi}(\rho_n)\dx\leq \frac{C_1}{2}\underset{\mathbb{R}^N}{\int}\varsigma (|u_n|^{N-1})\dx+ \frac{C_2}{2}\underset{\mathbb{R}^N}{\int}\varsigma(|u_n|^\nu{\Phi}_{1}(\alpha ^{\frac{N}{N-1}}|u_n|))\dx
\end{equation*}
 By using $\varsigma(t)=\text{max}\{t,t^N\}$. Now two cases arise: \\
 Case 1:
\begin{align}\label{3.18}
\underset{\mathbb{R}^N}{\int}\tilde{\Phi}_{1}(\rho_n)\dx &\leq C_1\left(\underset{\mathbb{R}^N}{\int}|u_n|^{N-1}\dx,\right)+C_2\left(\underset{\mathbb{R}^N}{\int}\left(|u_n|^{\nu}\right)\Phi_{1}\left(\alpha^\frac{N-1}{N}|u_n|\right)\dx\right),
\end{align}
 Case 2:
\begin{align}\label{3.19}
\underset{\mathbb{R}^N}{\int}\tilde{\Phi}_{1}(\rho_n)\dx &\leq C_1\left(\underset{\mathbb{R}^N}{\int}|u_n|^{N(N-1)}\dx\right) +C_2\left(\underset{\mathbb{R}^N}{\int}|u_n|^{N\nu}\Phi_{1}\left(\alpha^\frac{N-1}{N}|u_n|\right)\dx\right).
\end{align}
By H\"older's inequality and by Lemma \ref{lemma3.4}, we obtain
\begin{equation}\label{3.20}
\left(\underset{\mathbb{R}^N}{\int}\left(|u_n|^{\nu}\right)\Phi_{1}\left(\alpha^\frac{N-1}{N}|u_n|\right)\dx\right)\leq C_3,~~\forall n\in\mathbb{N}, 
\end{equation}
and 
\begin{equation}\label{3.20o}
\left(\underset{\mathbb{R}^N}{\int}\left(|u_n|^{N\nu}\right)\Phi_{1}\left(\alpha^\frac{N-1}{N}|u_n|\right)\dx\right)\leq C_4,~~\forall n\in\mathbb{N}. 
\end{equation}

Suppose case 1 holds, then by the definition of $ \varsigma$, $0<|u_n(x)|<1$. We have embedding
$$ \mathbf{X}_{\epsilon}\hookrightarrow W_{V_{\epsilon}}^{1,p}\hookrightarrow L^{q}(\mathbb{R}^N),~\forall q\in[p,p^*].$$
If $N-1<p^*$, we have
\begin{equation}\label{3.21} 
\|u_n\|_{N-1}^{N-1}\leq\|u_n\|_{\mathbf{X}_{\epsilon}}^{N-1}\leq C_5.
\end{equation}
and if $N-1>p^*$, we have
\begin{equation}\label{3.22}
    \|u_n\|_{N-1}^{N-1}\leq\|u_n\|_{p^*}^{p*}\leq C_6.
\end{equation}
For case 1, from equation \eqref{3.18}, \eqref{3.20}--\eqref{3.22}, we conclude that 
\begin{equation}\label{3.23}
   \underset{\mathbb{R}^N}{\int}\tilde{\Phi}_{1}(\rho_n)\dx\leq C_7,~~\forall n\in\mathbb{N.} 
\end{equation}
For case 2, from equation \eqref{3.19} and embedding of $\mathbf{X}_{\epsilon}$ in $L^s(\mathbb{R)^N}$ for $s\in[N,+\infty),$ we conclude that 
\begin{equation}\label{3.24}
    \underset{\mathbb{R}^N}{\int}\tilde{\Phi}_{1}(\rho_n)\dx\leq C_8,~\forall n\in\mathbb{N}. 
 \end{equation}
From equation \eqref{3.23} and \eqref{3.24}, it follows that $\{\rho_n\}_{n\geq1}$ is bounded in $L^{\tilde{\Phi}_1}(\mathbb{R}^N). $ So the sequence of functionals $\tilde{\rho}_n\subset\partial\Upsilon(u_n)\subset(E_{\Phi_1}(\mathbb{R}^N))^*$ corresponding to $\{\rho_n\}$ is also bounded in $(E_{\Phi_1}(\mathbb{R}^N))^*$. So there is $\tilde{\rho}_0\in (E_{\Phi_1}(\mathbb{R}^N))^*$ such that $\tilde{\rho}_n\overset{\ast}{\rightharpoonup}\tilde{\rho}_0$ in $(E_{\Phi_1}(\mathbb{R}^N))^*$ i.e.,
\begin{equation}\label{3.25}
\underset{\mathbb{R}^N}{\int} \rho_n v\dx=\langle \tilde{\rho}_n,v\rangle \rightarrow \langle \tilde{\rho}_0,v\rangle =\underset{\mathbb{R}^N}{\int} \rho_0 v\dx,~~\forall v\in\mathbf{X}_{\epsilon},
\end{equation}
for some $\rho_0\in L^{\tilde{\Phi}_{1}}(\mathbb{R}^N).$ From equation \eqref{3.15}, \eqref{3.16} and \eqref{3.25}, we have 
\begin{equation}\label{3.26}
  \sum_{t\in\{p,N\}}\underset{\mathbb{R}^N}{\int}|\nabla u_\epsilon|^{t-2}\nabla u_{\epsilon,\beta}\cdot\nabla v \dx ~+\underset{\mathbb{R}^N}{\int}V(x)| u_{\epsilon,\beta}|^{t-2} u_{\epsilon,\beta} v \dx=\underset{\mathbb{R}^N}{\int} \rho_0 v\dx, ~~\forall v\in \mathbf{X}_{\epsilon}.  
\end{equation}
Set $\upsilon=u_n-u_{\epsilon,\beta}$. Due to {Br\'ezis-Lieb-type results} (see \cite{Mahanta-2025}), we have
\begin{equation}\label{3.27}
\begin{cases}
\|u_n\|_{W_{V_{\epsilon}}^{1,p}}^p=\|u_{\epsilon,\beta}\|_{W_{V_{\epsilon}}^{1,p}}^p+\|\upsilon_n\|_{W_{V_{\epsilon}}^{1,p}}^p+o_n(1),\\
\|u_n\|_{W_{V_\epsilon}^{1,N}}^N=\|u_{\epsilon_\beta}\|_{W_{V_{\epsilon}}^{1,N}}^N+\|\upsilon_n\|_{W_{V_{\epsilon}}^{1,N}}^N+o_n(1).
\end{cases}
\end{equation}
So, 
\begin{equation}\label{3.28}
\|u_{\epsilon,\beta}\|_{W_{V_{\epsilon}}^{1,p}}^p+\|\upsilon_n\|_{W_{V_{\epsilon}}^{1,p}}^p+\|u_{\epsilon_\beta}\|_{W_{V_{\epsilon}}^{1,N}}^N+\|\upsilon_n\|_{W_{V_{\epsilon}}^{1,N}}^N-\underset{\mathbb{R}^N}{\int}\rho_nu_n\dx=o_n(1)
\end{equation}
and
\begin{equation}\label{3.29}
    \|u_{\epsilon,\beta}\|_{W_{V_{\epsilon}}^{1,p}}^p+\|u_{\epsilon_\beta}\|_{W_{V_{\epsilon}}^{1,N}}^N-\underset{\mathbb{R}^N}{\int}\rho_0u_{\epsilon,\beta}\dx=0.
\end{equation}
From equation \eqref{3.28} and \eqref{3.29}, we get
    \begin{equation*}
      o_n(1)= \|\upsilon_n\|_{W_{V_{\epsilon}}^{1,p}}^p+\|\upsilon_n\|_{W_{V_{\epsilon}}^{1,N}}^N-\underset{\mathbb{R}^N}{\int}\rho_nu_n \dx+\underset{\mathbb{R}^N}{\int}\rho_0u_{\epsilon,\beta}\dx. 
    \end{equation*}
    Now define $R:=\text{max}\{R_1,R_2\}$ where $R_1$ and $R_2$ are taken in \eqref{3.33} and \eqref{3.34}, respectively. For large enough $R>0$, we get
    \begin{equation}\label{3.30}
      o_n(1)= \|\upsilon_n\|_{W_{V_{\epsilon}}^{1,p}}^p+\|\upsilon_n\|_{W_{V_{\epsilon}}^{1,N}}^N-\underset{B_{R}}{\int}(\rho_nu_n-\rho_0u_{\epsilon,\beta})\dx-\underset{B_{R}^c}{\int}(\rho_nu_n-\rho_0u_{\epsilon,\beta})\dx.
    \end{equation}
    We have 
    \begin{equation*}
        \underset{B_{R}}{\int}(\rho_nu_n-\rho_0u_{\epsilon,\beta})\dx=\underset{B_{R}}{\int}\rho_n(u_n-u_{\epsilon,\beta})\dx+\underset{B_{R}}{\int}(\rho_n-\rho_0)u_{\epsilon,\beta}\dx
    \end{equation*} 
    Due to \eqref{3.25}, we deduce
    \begin{equation}\label{3.31}
      \underset{B_{R}}{\int}(\rho_n-\rho_0)u_{\epsilon,\beta}\dx\rightarrow0~\text{as}~n\rightarrow\infty.  
    \end{equation}
    Due to assumption \ref{g1}, H\"older's inequality, boundedness of $\{u_n\} ~\text{in}~{\mathbf{X}_{\epsilon}}$ and Compact Embedding of Sobolev spaces for a bounded domain, we have
    \begin{equation}\label{3.32}
        \underset{B_{R}}{\int}\rho_n(u_n-u_{\epsilon,\beta})\dx\rightarrow0 ~\text{as}~n\rightarrow\infty.
    \end{equation}
    Since $\{\rho_0u_{\epsilon,\beta}\}\in L^1(\mathbb{R}^N)$, so for large $R_1>0$ such that for any $\xi_1>0$, we have
    \begin{equation}\label{3.33}
    \underset{B_{R_1}^c}{\int}\rho_0u_{\epsilon,\beta}\dx<\xi_1.
     \end{equation}
      {From Lemma \ref{tight}, one can easily prove the tightness of the Cerami sequence $\{u_n\}_{n\geq 1}$ for $I_{\epsilon,\beta}.$} This implies that for any $\xi_2>0$, there is  $R_2>0$, and there holds
      \begin{equation}\label{3.34}
          \underset{B_{R_2}^c}{\int}\rho_nu_n\dx<\xi_2.
      \end{equation}
      From equations \eqref{3.30}--\eqref{3.34}, it follows that 
      \begin{equation*}
\|\upsilon_n\|_{W_{V_{\epsilon}}^{1,p}}^p+\|\upsilon_n\|_{W_{V_{\epsilon}}^{1,N}}^N\rightarrow0~\text{as}~n\rightarrow+\infty  .
      \end{equation*}
      This implies $\upsilon_n\rightarrow0$ in $\mathbf{X}_{\epsilon}$. Hence, $u_n\rightarrow u_{\epsilon,\beta}$ in $\mathbf{X}_{\epsilon}$. So,
      \begin{equation}\label{3.35}
      0\in\partial I_{\epsilon,\beta}(u_{\epsilon,\beta}).
      \end{equation}
      From equations \eqref{3.29}, \eqref{3.35} and Lemma \ref{lemma2.4}, it follows that $u_{\epsilon,\beta}$ is a non-trivial weak solution for \eqref{Auxiliary problem}.
       \end{proof}
       
     \section{Autonomous Problem}\label{section4}
     In this section, we consider an autonomous problem related to \eqref{main problem} and compare the critical level of both problems. We are motivated by the methods discussed in \cite[Fiscella and Pucci]{Fiscella-2021} and \cite[Costa and Figueiredo]{Costa-F-2022}. The autonomous problem related to \eqref{main problem} is as follows:
     \begin{equation}\label{Autonomous problem}
    \left\{
     \begin{aligned}
       &-\Delta_pu-\Delta_Nu~+V_0(|u|^{p-2}u+|u|^{N-2}u)=f(u)~~ \text{in}~~ \mathbb{R}^N,\\&
       ~u\in W^{1,p}_{V_0}({\mathbb{R}^N})\cap W^{1,N}_{V_0}({\mathbb{R}^N}).
    \end{aligned}\tag{$\mathscr{P}_{V_0}$}
    \right.
\end{equation}
Denote $\boldsymbol{Y}=W^{1,p}_{V_0}({\mathbb{R}^N})\cap W^{1,N}_{V_0}({\mathbb{R}^N}) $, with norm 
\begin{equation*}
    \|u\|_{\boldsymbol{Y}}=\|u\|_{W^{1,p}_{V_0}}+\|u\|_{W^{1,N}_{V_0}}
\end{equation*}
and $F(t)=\underset{0}{\int^t}f(s)\mathrm{d}s.$\\
The energy functional associated with \eqref{Autonomous problem} is defined as $I_{V_0}:W^{1,p}_{V_0}({\mathbb{R}^N})\cap W^{1,N}_{V_0}({\mathbb{R}^N})\rightarrow\mathbb{R} $
\begin{equation*}
    I_{V_0}(u)=\frac{1}{p}\underset {\mathbb{R}^N}{\int}(|\nabla u|^p+V_0|u|^p)\dx+\frac{1}{N}\underset {\mathbb{R}^N}{\int}(|\nabla u|^N+V_0|u|^N)\dx-\underset {\mathbb{R}^N}{\int}F(u)\dx.
\end{equation*}
Hence 
\begin{equation*}
   I_{V_0}(u) =\frac{1}{p}\|u\|_{W^{1,p}_{V_0}}^p+\frac{1}{N}\|u\|_{W^{1,N}_{V_0}}^N-\underset {\mathbb{R}^N}{\int}F(u)\dx.
\end{equation*}
$I_{V_0}$ is well defined and belongs to $ C^1(\boldsymbol{Y},\mathbb{R}).$ { Using the similar strategy in \cite[Lemma 2]{Lam-2013} and by invoking the Mountain Pass theorem for $C^1$ functional, we  obtained a Cerami sequence $\{\hat{u}_n\}_{n\geq1}\subset\boldsymbol{Y}$ i.e.
\begin{equation}\label{4.1}
    I_{V_0}(\hat{u}_n)\rightarrow c_{V_0}~~\text{and}~\|1+\hat{u}_n\|_{\boldsymbol{Y}
    }~I_{V_0}'(\hat{u}_n)\rightarrow0~~\text{as}~n\rightarrow+\infty.
\end{equation}
where $$c_{V_0}=\underset{\gamma\in \Gamma}{\inf}\underset{t\in [0,1]}{\max}I_{V_0}(\gamma(t))$$
and 
$$ \Gamma=\{\gamma\in C([0,1],\boldsymbol{Y}):\gamma(0)=0~,I_{\boldsymbol{V_0}}(\gamma(1))<0\}.$$}
By adapting the strategy found in {\cite[Lemma 5]{Lam-2013}}, one can prove the boundedness of $\{\hat{u}_n\}_{n\geq1}$ in $\boldsymbol{Y}$. So up to a subsequence still denoted by itself, we can assume that 
\begin{equation*}
     \begin{cases}
     \hat{u}_n\rightharpoonup u_0~\text{in~}~\boldsymbol{Y},\\
     \hat{u}_n(x)\rightarrow u_{0}(x)~\text{a.e. in~}~\mathbb{R}^N.
     \end{cases}
     \end{equation*}
    \subsection{Existence of Critical Point} In the following Lemma, we will show the existence of critical point for the functional corresponding to \ref{Autonomous problem}.
\begin{lemma}
    Let the assumption \ref{V1} and \ref{f1}-\ref{f5} hold. Then the function $u_0$ is a non-trivial critical point for the $I_{V_0}.$ \end{lemma}
    \begin{proof}
      As $n \to \infty$, from \eqref{4.1}, it follows that 
      \begin{equation}\label{4.2}
          I_{V_0}(\hat{u}_n)=c_{V_0}+o_n(1) ~\text{and}~\langle I_{V_0}'(\hat {u}_n),\hat{u}_n\rangle=o_n(1).
      \end{equation}
     Using arguments as in Section \ref{section3}, as $n \to \infty$, for $v\in \boldsymbol{Y}$, it implies
      \begin{equation*}
          \nabla \hat u_n\rightarrow\nabla  u_0~~\text{a.e.}~\text{in~}\mathbb{R}^N
      \end{equation*}
     \begin{equation}\label{4.3}
         \sum_{t\in\{p,N\}}\underset{\mathbb{R}^N}{\int}|\nabla\hat{u}_n|^{t-2}\nabla\hat{u}_n\cdot\nabla v \dx\rightarrow\sum_{t\in\{p,N\}}\underset{\mathbb{R}^N}{\int}|\nabla{u}_0|^{t-2}\nabla{u_0}\cdot\nabla v \dx
     \end{equation}
     and 
     \begin{equation}\label{4.4}
      \sum_{t\in\{p,N\}}\underset{\mathbb{R}^N}{\int}V_0|\hat{u}_n|^{t-2}\hat{u}_n v \dx\rightarrow    \sum_{t\in\{p,N\}}\underset{\mathbb{R}^N}{\int}V_0|{u}_0|^{t-2}{u}_0 v \dx
     \end{equation}
     {\bf Claim:} $\underset{\mathbb{R}^N}{\int}f(\hat{u}_n)v \dx\rightarrow \underset{\mathbb{R}^N}{\int}f({u}_0)v \dx$ for all $v\in\boldsymbol{Y.}$\\
     Due to assumption \ref{f1}, \ref{f2}, Remark \ref{rem1.1} and Lemma \ref{lemma3.4}, the sequence $\{f(\hat{u}_n)v\}_{n\geq1}$ is bounded in $L^1(\mathbb{R}^N).$ This also implies that the sequence $\{f(\hat{u}_n)v\}$ is uniformly integrable over $\mathbb{R}^N$ and $f(\hat{u}_n)v\rightarrow f(u_0)v$ a.e. in $\mathbb{R}^N.$ Let $R>0$ such that $\psi_R\in C_c^{\infty}(\mathbb{R}^N)$, $0\leq\psi_R\leq1$ in $\mathbb{R}^N$  and
      \begin{equation}
       \psi_{R}\equiv0~\text{in}~B_{\frac{R}{2}} \text{ and }
       \psi_{R}\equiv1~\text{in}~B_{R}^c,
\end{equation}
with $|\nabla\psi_R|\leq\frac{C}{R}$, where $C$ is constant independent of $R.$ Consider $U=v\psi_R$, from equation \eqref{4.2}, as $n \to \infty,$ we have
\begin{align*}
\bigg|\underset{\mathbb{R}^N}{\int}f(\hat{u}_n)v\psi_R \dx \bigg|&\leq\bigg|\sum_{t\in\{p,N\}}\underset{\mathbb{R}^N}{\int}|\nabla\hat{u}_n|^{t-2}\nabla\hat{u}_n\cdot\nabla (v\psi_R) \dx+   \sum_{t\in\{p,N\}}\underset{\mathbb{R}^N}{\int}V_0|\hat{u}_n|^{t-2}\hat{u}_n( v\psi_R)\dx\bigg|\\ &\quad+o_n(1) \\
&=\bigg|\sum_{t\in\{p,N\}}\underset{\mathbb{R}^N}{\int}|\nabla\hat{u}_n|^{t-2}\nabla\hat{u}_n\cdot(\psi_R\nabla v+v\nabla\psi_R) \dx\\&\quad+   \sum_{t\in\{p,N\}}\underset{\mathbb{R}^N}{\int}V_0|\hat{u}_n|^{t-2}\hat{u}_n( v\psi_R)\dx\bigg|+o_n(1)\\
&\leq\sum_{t\in\{p,N\}}\underset{\mathbb{R}^N}{\int}|\nabla\hat{u}_n|^{t-2}|\nabla\hat{u}_n\cdot\nabla v|\dx+\frac{C}{R}\sum_{t\in\{p,N\}}\underset{\mathbb{R}^N}{\int}|\nabla\hat{u}_n|^{t-2}|\nabla\hat{u}_n||v| \dx\\&\quad+   \sum_{t\in\{p,N\}}\underset{\mathbb{R}^N}{\int}V_0|\hat{u}_n|^{t-2}|\hat{u}_n||v| \dx+o_n(1).
\end{align*}
Since $v\in\boldsymbol{Y}$, so for large enough $R>0$, for any given $\epsilon_1>0$ ~\text{and} $\epsilon_2>0$ we have $$\underset{B_R^c}{\int}|\nabla v|^t\dx<\epsilon_1~
\text{and} \underset{B_R^c}{\int}|v|^t\dx<\epsilon_2.
$$ Now using boundedness of $\{\hat{u}_n\}$, we have $\underset{B_R^c}{\int }f(\hat{u}_n)vdx<\xi^*$ for large enough $R>0$ and any $\xi^*>0$. 
This concludes that $\{f(\hat{u}_n)v\}_{n\geq1}$ is tight over $\mathbb{R}^N.$ Hence, by Vitali's Convergence theorem,
\begin{equation}\label{4.6}
\underset{\mathbb{R}^N}{\int}f(\hat{u}_n)v\dx\rightarrow \underset{\mathbb{R}^N}{\int}f({u}_0)v\dx~
~\text{for all}~v\in\boldsymbol{Y.}
\end{equation}
From equation \eqref{4.3},\eqref{4.4} and \eqref{4.6} it follows that $u_0$ is critical point of $I_{V_0} $ and using similar arguments in Lemma \ref{lemma3.5}, we can prove that $\hat{u}_n\rightarrow u_0~\text{in}~\boldsymbol{Y.}$ Hence, $u_0$ is non trivial critical point.
    \end{proof}
 \subsection{Relationship Between Both Mountain Pass Levels}
In the next lemma, we establish a relation between $c_{\epsilon,\beta}$ and $c_{V_0}$, which plays a crucial role in our arguments for proving the main theorem.
\begin{lemma}\label{lemma4.2}
Assume that \ref{V1} and \ref{f1}-\ref{f5} hold. Then $$\underset{\epsilon,\beta\rightarrow0}{{\lim}}c_{\epsilon,\beta}=c_{V_0}$$ where $c_{\epsilon,\beta}$ and $c_{V_0}$ are mountain pass level for \eqref{main problem} and \eqref{Autonomous problem}, respectively. \end{lemma}
\begin{proof}
    Due to assumption \ref{V1} and $H(t)\leq1$, we obtain 
    \begin{equation}\label{n4.7}
    \underset{\epsilon,\beta\rightarrow0}{{\liminf}}~c_{\epsilon,\beta}\geq~c_{V_0}.
    \end{equation}
    Next, we will prove that 
    \begin{equation}\label{n4.8}
        \underset{\epsilon,\beta\rightarrow0}{{\limsup}}~c_{\epsilon,\beta}\leq~c_{V_0}.
        \end{equation}
     Let $u_0$ be the solution of \eqref{Autonomous problem} with $I_{V_0}(u_0)=c_{V_0}$ and $I'(u_0)=0.$ Consider $\phi\in C_c^{\infty}(\mathbb{R}^N)$ such that $0\leq\phi\leq1$ in $\mathbb{R}^N$ and
     \begin{equation*}
       \phi\equiv1,~x\in~B_1  \text{ and }
       \phi\equiv0,~x\in B_2^c.
\end{equation*}
Let $R>0$, define $u_R(x)=\phi(\frac{x}{R})u_0$ and $B_{2R}\subset \Lambda_{\epsilon.} $ Since $|u_{R}|\leq u_0$, so by virtue of the Lebesgue dominated convergence theorem, it follows that 
\begin{equation}
    u_R\rightarrow u_0~~\text{in}~W^{1,t}(\mathbb{R}^N)~\text{as}~ R\rightarrow+\infty~ \text{for}~t\in \{p,N\}.
\end{equation}
It follows that $u_R\rightarrow u_0$ a.e. in $\mathbb{R}^N.$ Define $h:[0,\infty)\rightarrow\mathbb{R}$, $$h(t)=I_{V_0}(tu_R).$$
It can be proved that $h(t)>0$ for sufficiently small $t$ and $h(t)<0$ for large $t.$ As $h$ is continuously differentiable function. So there exist $t_R>0$ such that $h'(t_R)=I_{V_0}(t_Ru_R)u_{R}=0$. This implies that for every $R>0$ there exist $t_R>0$ such that $t_{R}u_{R}\in \mathcal{N}_0.$ Now we claim the following.\\
{\bf Claim}: The sequence $\{t_R\}$ is bounded and upto a subsequence still denoted by itself we have, $$\underset{R\rightarrow+\infty}{\text{lim}}t_{R}=1.$$\\
Suppose, on the contrary, that $t_R\rightarrow+\infty$ as $R\rightarrow+\infty.$ Since $t_{R}u_{R}\in \mathcal{N}_0,$ this  implies that 
\begin{equation}
    t_{R}^p\|u_R\|_{{W_{V_0}^{1,p}}}^p+ t_{R}^N\|u_R\|_{{W_{V_0}^{1,N}}}^N=\underset{\mathbb{R}^N}{\int}f(t_Ru_R)t_Ru_R \dx
\end{equation}
Using assumptions \ref{f2} and \ref{f5}, we obtain
\begin{equation*}
  t_{R}^p\|u_R\|_{{W_{V_0}^{1,p}}}^p+ t_{R}^N\|u_R\|_{{W_{V_0}^{1,N}}}^N=\underset{\mathbb{R}^N}{\int}f(t_Ru_R)t_Ru_R \dx\geq\underset{\mathbb{R}^N}{\int}NF(t_Ru_R)\dx.
\end{equation*}
That is,
\begin{equation}
  \frac{1}{t_R^{N-p}}\|u_R\|_{{W_{V_0}^{1,p}}}^p+ \|u_R\|_{{W_{V_0}^{1,N}}}^N\geq N\xi t_R\underset{\mathbb{R}^N}{\int}|u_R|^{N+1}\dx
\end{equation}
As $R\rightarrow+\infty$, we obtain $\|u_0\|_{W_{V_0}^{1,N}}^N\geq+\infty$, which is not possible. Hence $\{t_R\}$ is bounded.\\
Again suppose that $\underset{R\rightarrow+\infty}{\text{lim}}t_R\neq1$. This means either $\underset{R\rightarrow+\infty}{{\lim}}t_R=t_1>1$ or $\underset{R\rightarrow+\infty}{\text{lim}}t_R=t_1<1$. As $R\rightarrow+\infty$, we get
\begin{equation}\label{eq4.15}
   \frac{1}{t_1^{N-p}}\|u_0\|_{{W_{V_0}^{1,p}}}^p+ \|u_0\|_{{W_{V_0}^{1,N}}}^N=\underset{\mathbb{R}^N}{\int}f(t_1u_0)t_1u_0 \dx 
\end{equation}
Since $u_0\in \mathcal{N}_0$, it follows that 
\begin{equation}\label{eq4.16}
   \|u_0\|_{{W_{V_0}^{1,p}}}^p+ \|u_0\|_{{W_{V_0}^{1,N}}}^N=\underset{\mathbb{R}^N}{\int}f(u_0)u_0\dx. 
\end{equation}
Now ,by subtracting the \eqref{eq4.16} from \eqref{eq4.15}, we get
\begin{equation*}
\begin{split}
    \left(\frac{1}{t_1^{N-p}}-1\right) \|u_0\|_{{W_{V_0}^{1,p}}}^p=\underset{\mathbb{R}^N}{\int}\left(\frac{1}{t_1^{N}}f(t_1u_0)t_1u_0-f(u_0)u_0\right)\dx\\
    =\underset{\mathbb{R}^N}{\int}\left(\frac{1}{t_1^{N-1}}f(t_1u_0)-f(u_0)\right)u_0 \dx
    \end{split}
\end{equation*}
Using assumption \ref{f4}, it follows that in both cases we obtained that $\|u_0\|_{{W_{V_0}^{1,p}}}^p<0$,
which is not possible. Hence, the claim follows, and this implies that $t_Ru_R\rightarrow u_0$ as $R\rightarrow+\infty$  in $\boldsymbol{Y}.$
Similar to Lemma \ref{lemma 3.1}, there is $\hat{t}>0$ such that $I_{\epsilon,\beta}(\hat{t}t_{R}u_{R})<0$. Define $g(t)=t\hat{t}t_Ru_R$ for $t\in[0,1].$ Clearly $g\in\Gamma_{\epsilon,\beta}$. Hence, it holds
\begin{equation}\label{4.18}
    c_{\epsilon,\beta}\leq\underset{t\in[0,1]}{{\max}} I_{\epsilon,\beta}(g(t))\leq\underset{t\geq0}{{\max}}I_{\epsilon,\beta}(\hat{t}tt_Ru_R)=I_{\epsilon,\beta}(t_*t_Ru_R), 
\end{equation}
for some $t_*=t_*(\epsilon,\beta,R).$ Without loss of generality, we can take $V(0)=V_0$. Due to which $V(\epsilon x)\rightarrow V_0$ as $\epsilon\rightarrow0$ i.e. for any $\epsilon'>0$, there is $\epsilon_0>$  and a ball centered at $0$ say ($B_{2R}$ in $\mathbb{R}^N$ )such that
\begin{equation}
    |V(\epsilon x)-V_0|<\epsilon',\ \forall \ \epsilon\in(0,\epsilon_0)~ \text{and} ~x \in B_{2R}.
\end{equation}
The above equation implies that $V(\epsilon x)<\epsilon'+V_0$.
From equation \eqref{4.18}, we get
\begin{align*}
c_{\epsilon,\beta}&=\frac{1}{p}\underset{\mathbb{R}^N}{\int}\left(t_*^pt_R^p|\nabla u_R|^p+V(\epsilon x)t_*^pt_R^p|u_R|^p\right)\dx+\frac{1}{N}\underset{\mathbb{R}^N}{\int}\left(t_*^Nt_R^N|\nabla u_R|^N+V(\epsilon x)t_*^Nt_R^N|u_R|^N\right) \dx\\ &\quad-\underset{\mathbb{R}^N}{\int}G_{H}(\epsilon x,t_*t_Ru_R)\dx\\
&\leq\sum_{s\in\{p,N\}}\frac{1}{s}\underset{\mathbb{R}^N}{\int}\left(t_*^st_R^s|\nabla u_R|^s+V_0t_*^st_R^s|u_R|^s\right)\dx+\sum_{s\in\{p,N\}}\frac{\epsilon't_*^st_R^s}{s}\underset{\mathbb{R}^N}{\int}|u_R|^s \dx\\&\quad -\underset{\mathbb{R}^N}{\int}H(t_*t_Ru_R-\beta)G(\epsilon x,t_*t_Ru_R)\dx.
\end{align*}
For $\beta\rightarrow0$, we have
\begin{align*}
  c_{\epsilon,\beta}\leq I_{V_0}{(t_*t_Ru_R)}+\epsilon'\sum_{s\in\{p,N\}}\frac{t_*^st_R^s}{s}\underset{\mathbb{R}^N}{\int}|u_R|^s \dx.
\end{align*}
Next, for large enough $R>0$, the above equation leads to $\underset{\epsilon,\beta \rightarrow0}{{\limsup}}~c_{\epsilon,\beta}\leq c_{V_0}.$
Therefore, from \eqref{n4.7} and \eqref{n4.8}, we can conclude that the limit exists and $\underset{\epsilon,\beta\rightarrow0}{\text{lim}}c_{\epsilon,\beta}=c_{V_0}.$
\end{proof}
\subsection{Compactness Result}
In the next lemma, we will prove the compactness result using the Lions Compactness principle \cite{Lions-1984}.
\begin{lemma}\label{lemma4.3}
   Assume that \ref{V1}-\ref{V2} and \ref{f1}-\ref{f5} hold. Let $\epsilon_n$ and $\beta_n\rightarrow0$ as $n\rightarrow\infty$ and $u_n=\{u_{\epsilon_n,\beta_n}\}_{n\geq 1}\subseteq\mathbf{X}_{\epsilon}$ be a non negative sequence such that $0\in\partial I_{\epsilon_n,\beta_n}(u_n)$, $I_{\epsilon_n,\beta_n}(u_n)\rightarrow c_{V_0}$ and $\underset{n\rightarrow+\infty}{{\limsup}}\|u_n\|_{W^{1,N}(\mathbb{R}^N)}^{N'}<\frac{\alpha_N}{\alpha_0}$. Then there exits a sequence $\{y_n\}_{n\geq 1}\subseteq\mathbb{R}^N$ such that sequence 
    $$ w_n(x)=u_n(x+y_n)$$
    has a convergent subsequence in $\boldsymbol{Y}.$ In addition, upto a subsequence $a_n=\{\epsilon_ny_n\}\rightarrow y_0$ as $n\rightarrow+\infty$ for some $y_0\in \Lambda$ and $V(y_0)=V_0.$ \end{lemma}
    \begin{proof}
       From lemma \ref{lemma 3.2} and assumption \ref{V1}, it follows that $\{u_n\}_{n\geq 1}$ is bounded in $\boldsymbol{Y}$.\\ 
       {\bf Claim 1:} There exists $\{y_n\}\subseteq\mathbb{R}^N$, $R>0$ and $\hat{\sigma}>0$ such that $\underset{n\rightarrow+\infty}{{\liminf}}\underset{B_{R}(y_n)}{\int}|u_n|^N \dx\geq\hat{\sigma}.$\\
       Suppose that the above claim doesn't hold. This means
       $\underset{n\rightarrow+\infty}{{\liminf}}~\underset{y\in\mathbb{R}^N}{{\sup}}\underset{B_{R}(y)}{\int}|u_n|^N \dx=0. $
        From \cite{Lions-1984, Alves-2009}, it follows that $u_n\rightarrow0$ in $L^{\sigma}(\mathbb{R}^N)$ for any $\sigma\in(N,\infty).$ Since $0\in\partial I_{\epsilon_n,\beta_n}(u_n) $, this implies that 
       \begin{equation*}
\|u_n\|_{W_{V_{\epsilon_n}}^{1,p}}^p+\|u_n\|_{W_{V_{\epsilon_n}}^{1,N}}^N=\underset{\mathbb{R}^N}{\int}\rho_nu_n \dx+o_n(1)~~\text{as}~n\rightarrow+\infty.
       \end{equation*}
       From assumption \ref{f1}, \ref{f2} (Remark \ref{rem1.1}) and $\underset{n\rightarrow+\infty}{{\limsup}}\|u_n\|_{W^{1,N}(\mathbb{R}^N)}^{N'}<\frac{\alpha_N}{\alpha_0}$, it follows that for any $\tau>0$
       \begin{equation*}
           \left| \underset{\mathbb{R}^N}{\int}\rho_nu_n \dx \underset{}{}\right|\leq M_1\tau,~\text{as}~n\rightarrow+\infty.
       \end{equation*}
       So for sufficiently small $\tau>0$ we deduce that $\|u_n\|_{W_{V_{\epsilon_n}}^{1,p}}^p+\|u_n\|_{W_{V_{\epsilon_n}}^{1,N}}^N\rightarrow0~~\text{as}~n\rightarrow+\infty.$
       This means $u_n\rightarrow 0$ in $\boldsymbol{X_{\epsilon_n}}$ as $n\rightarrow+\infty.$ It follows that $I_{\epsilon_n,\beta_n}(u_n)\rightarrow0$ as $n\rightarrow+\infty$ which contradicts lemma \ref{lemma4.2}. Hence, our claim follows.\\
       Define $w_n(x)=u_n(x+y_n)$. Since $\|.\|_{\boldsymbol{Y}}$ is invariant under translation. This implies that $\{w_n\}_{n\geq1}$ is bounded in $\boldsymbol{Y}$. So up to a subsequence 
 $$\begin{cases}
               w_n\rightharpoonup~w~\text{in}~\boldsymbol{Y},\\
               w_n\rightarrow w~\text{in}~L^s(B_R)~\text{for}~s\geq1,\\
               w_n(x)\rightarrow~w(x)~\text{a.e. in}~ \mathbb{R}^N.
           \end{cases}$$
           From Claim 1, there holds $               \underset{B_R(0)}{\int}|w|^Ndx\geq\hat{\sigma}>0.$
           This means $w\not\equiv0$.\\ Next, we will show that $\{a_n\}=\{\epsilon_ny_n\}_{n\geq1}$ is bounded in $\mathbb{R}^N$. If we are able to show that $\underset{n\rightarrow+\infty}{{\lim}}\text{dist}(a_n,\overline{\Lambda})=0,$
        then we can deduce boundedness of $\{a_n\}_{n\geq1}.$
        On the contrary, we assume that it is not true, then there exists $\delta^*>0$ and a subsequence of $\{a_n\}$ such that 
         $\mathrm{dist}(a_n,\overline{\Lambda})\geq \delta^*$ for all $n\in \mathbb{N}.$
        This implies there exists $r>0$, such that $B_r(a_n)\subset\Lambda^c.$ Define $\phi_n(x)=\psi(\frac{x}{n})w(x)$, where $\psi$ is defined in lemma \ref{lemma 3.3}. Note that, $\phi_n\rightarrow w$ in $\boldsymbol{Y}$ as $n\rightarrow+\infty.$ Since, $0\in\partial I_{\epsilon_n,\beta_n}(u_n)$ and assumption \ref{V1} with a change of variable $z\mapsto x+y_n$, we obtain
        \begin{equation}\label{4.25}
        \begin{split}
            \sum_{t\in\{p,N\}}\underset{\mathbb{R}^N}{\int}(|\nabla w_n|^{t-2}\nabla{w}_n\nabla\phi_n +V_0|w_n|^{t-1}\phi_n) \dx\leq\underset{\mathbb{R}^N}{\int}g(\epsilon_nx,w_n)\phi_n \dx\\
            =\underset{B_{\frac{r}{\epsilon_n}}}{\int}g(\epsilon_nx,w_n)\phi_n \dx+\underset{B_{\frac{r}{\epsilon_n}}^c}{\int}g(\epsilon_nx,w_n)\phi_n \dx.
            \end{split}
        \end{equation}
        Since $m(B_{\frac{r}{\epsilon_n}}^c)\rightarrow0~~\text{as}~n\rightarrow+\infty$, by the Lebesgue dominated convergence theorem 
        $$\underset{B_{\frac{r}{\epsilon_n}}^c}{\int}g(\epsilon_nx,w_n)\phi_n \dx\rightarrow0~~ \text{as}~~ n\rightarrow+\infty.$$
        Now 
        \begin{equation*}
          \underset{B_{\frac{r}{\epsilon_n}}}{\int}g(\epsilon_nx,w_n)\phi_n \dx=\underset{B_{\frac{r}{\epsilon_n}}}{\int}\tilde{f}(w_n)\phi_n \dx\leq\underset{B_{\frac{r}{\epsilon_n}}}{\int}f(w_n)\phi_n\dx+\frac{V_0}{k}\underset{B_{\frac{r}{\epsilon_n}}}{\int}|w_n|^{N-1}\phi_n \dx. 
        \end{equation*}
        Using H\"older's inequality one can show that $$\frac{V_0}{k}\underset{B_{\frac{r}{\epsilon_n}}}{\int}|w_n|^{N-1}\phi_n \dx \rightarrow 0 \text{ as } n\rightarrow+\infty.$$ From assumptions \ref{f1}, \ref{f2} and $\underset{n\rightarrow+\infty}{{\limsup}}\|u_n\|_{W^{1,N}(\mathbb{R}^N)}^{N'}<\frac{\alpha_N}{\alpha_0}$, we get $$\underset{B_{\frac{r}{\epsilon_n}}}{\int}f(w_n)\phi_n \dx\rightarrow0~\text{as}~n\rightarrow+\infty.$$
        Therefore, as $n\rightarrow+\infty$ it follows that $w\equiv0$, which is absurd. Hence $\{a_n\}$ is bounded in $\mathbb{R}^N.$ So up to a subsequence, still denoted by itself, such that $a_n\rightarrow y_0\in\overline{\Lambda}.$
        
        {\bf Claim 2:} The sequence $w_n\rightarrow w~\text{in}~\boldsymbol{Y}$.\\  For each $n\in\mathbb{N}$, there is a $t_n>0$ such that $t_nw_n \in\mathcal{N}_0$ and using similar arguments as in Lemma \ref{lemma4.2}, it can be proved that $\{t_n\}_{n\geq1}$ is bounded and up to a subsequence we can assume that $t_nw_n\rightarrow t_0w$ a.e. in $\mathbb{R}^N$. So by a change of variable, \ref{V1} and  using $g_H(x,t)\leq f(t)$, we get, 
        \begin{equation}
            c_{V_0}\leq I_{V_0}(t_nw_n)\leq I_{\varepsilon_n,\beta_n}(t_nu_n)\leq I_{\epsilon_n,\beta_n}(u_n)=c_{V_0}+o_n(1)~\text{as}~n\rightarrow+\infty
        \end{equation}
    Consequently, $I_{V_0}(t_nw_n)\rightarrow c_{V_0}$ as $n\rightarrow+\infty.$ Now using Lemma \ref{lemma2.3}, we conclude our claim.\\
    {\bf Claim 3:} $V(y_0)=V_0~~\text{and}~y_0\in\Lambda.$\\
        Suppose that the above claim is not true. This implies that $V(y_0)>V_0$. Once that $t_nw_n\rightarrow t_0w$ in $\boldsymbol{Y}$, by Fatou's lemma with the invariance of $\mathbb{R}^N$ by translation we obtain that
        \begin{align*}
               c_{V_0}&=I_{V_0}(t_0w)<\sum_{s\in\{p,N\}}\frac{1}{s}(|\nabla t_0w|^s+V(y_0)|t_0w|^s)\dx-\underset{\mathbb{R}^N}{\int}F(t_0w)\dx\\
               &\leq\underset{n\rightarrow+\infty}{{\liminf}}\sum_{s\in\{p,N\}}\frac{1}{s}(|\nabla t_nw_n|^s+V(\epsilon_nx+a_n)|t_nw_n|^s) \dx-\underset{\mathbb{R}^N}{\int}F(t_nw_n)\dx\\
               &\leq\underset{n\rightarrow+\infty}{{\liminf}}~I_{\epsilon_n,\beta_n}(t_nu_n)\leq c_{V_0}.
        \end{align*}
        which is not possible. Therefore $V(y_0)=V_0$ and thanks to \ref{V2}, it follows that $y_0\notin\partial\Lambda$. Hence $y_0\in\Lambda.$
    \end{proof}
    \subsection{Moser Iteration Argument}
In the next lemma, we will show the $
L^{\infty}$ estimate of $\{w_n\}_{n\geq1}$ using Moser iteration arguments.
\begin{lemma}\label{lemma4.4}
  Assume that \ref{V1}-\ref{V2} and \ref{f1}-\ref{f5} hold. Let $ \{w_n\}_{n\geq1}$ be  defined in lemma \ref{lemma4.3}. Then there exist a constant $K>0$ such that $$\|w_n\|_{L^{\infty}(\mathbb{R}^N)}\leq K, \text{ for all } n\in\mathbb{N}.$$
  {Moreover, we have $$\underset{|x|\rightarrow+\infty}{{\lim}}w_n(x)=0,~\text{uniformly in } n\in\mathbb{N.}$$} \end{lemma}
  \begin{proof}
      From Lemma \ref{lemma4.2}, it follows that $\underset{n\rightarrow+\infty}{\lim}I_{\epsilon_n,\beta_n}(u_n)=c_{V_0}$. Furthermore, there exist $\{y_n\}_{n\geq 1}\subseteq\mathbb{R}^N$ satisfies $w_n(x)=u_n(x+y_n)$ and $\{\epsilon_ny_n\}\rightarrow y_0\in\Lambda.$ For  $L>0~\text{and}~\gamma>1$, define
      \begin{equation}
          g(w_n)=w_nw_{L,n}^{N(\gamma-1)}~\text{and}~v_{L,n}=w_nw_{L,n}^{\gamma-1}
      \end{equation}
      where $w_{L,n}=\text{min}\{w_n,L\}.$ On using $g(w_n)$ as a test function in equation \eqref{3.26} and  a simple change of variable leads to
      \begin{equation*}
           \sum_{s\in\{p,N\}}\underset{\mathbb{R}^N}{\int}\big[|\nabla w_n|^{s-2}\nabla w_n\nabla g(w_n)+V(\epsilon_nx+a_n)|w_n|^{s-2}w_ng(w_n)\big]\dx=\underset{\mathbb{R}^N}{\int}\rho_n(x+y_n)g(w_n)\dx, 
           \end{equation*}
           then
           \begin{align*}
          \sum_{s\in\{p,N\}}\underset{\mathbb{R}^N}{\int}(|\nabla w_n|^{s} w_{L,n}&+N(\gamma-1)|\nabla w_n|^{s-2}\nabla w_n\nabla w_{L,n}w_nw_{L,n}^{N(\gamma-1)-1}\dx\\ &+V(\epsilon_nx+a_n)|w_n|^{s-2}w_ng(w_n))\dx=\underset{\mathbb{R}^N}{\int}\rho_n(x+y_n)g(w_n)\dx.
      \end{align*}
      Therefore, we have 
      \begin{align}\label{4.31}
        \sum_{s\in\{p,N\}}\underset{\mathbb{R}^N}{\int}(|\nabla w_n|^{s} w_{L,n})\dx&+N(\gamma-1) \underset{w_n\leq L}{\int}w_nw_{L,n}^{N(\gamma_-1)-1}|\nabla w_{L,n}|^s\dx \notag\\&\quad + \underset{\mathbb{R}^N}{\int}V(\epsilon_nx+a_n)|w_n|^{s-2}w_ng(w_n)\dx=\underset{\mathbb{R}^N}{\int}\rho_n(x+y_n)g(w_n)\dx .
      \end{align}
      Set 
      \begin{equation}\label{4.32}
          \Xi(t)=\underset{0}{\int}^t(g'(\tau))^\frac{1}{N}\mathrm{d}\tau~\text{and}~\theta(t)=\frac{{|t|}^{N}}{N}.
      \end{equation}
    since $g$ is an increasing function, therefore $(s-t)(g(s)-g(t))\geq 0$ for all $s,t\in\mathbb{R}.$
      Further, applying Jensen's inequality, one has 
      \begin{equation}
          |\theta(s)-\theta(t)|\leq\Xi'(s-t)(g(s)-g(t)),~\forall~s,t\in\mathbb{R}.
      \end{equation}
      From equation \eqref{4.32}, it follows that
      \begin{equation}
          \frac{1}{\gamma}w_{L,n}^{\gamma-1}w_n\leq\Xi(w_n)\leq w_{L,n}^{\gamma-1}w_n.
      \end{equation}
      Due to equivalent norm and  continuous embedding $W_{V_0}^{1,N}(\mathbb{R}^N)\hookrightarrow L^s(\mathbb{R}^N)$ for $s>N$, i.e., there is a best constant $C_*>0$ such that 
      \begin{equation}\label{4.36}
        \frac{1}{\gamma^N}C_*^N\|w_nw_{L,n}^{\gamma-1}\|_{L^s(\mathbb{R}^N)}^N\leq C_{*}^N\|\Xi(w_n)\|_{L^s(\mathbb{R}^N)}^N\leq \|\Xi(w_n)\|_{W_{V_0}^{1,N}(\mathbb{R}^N)}^N.
      \end{equation}
      Then, \eqref{4.31} and assumptions \ref{f1}, \ref{f2} (Remark \ref{rem1.1}) lead us to the following
      \begin{equation}
       \|\Xi(w_n)\|_{W_{V_0}^{1,N}(\mathbb{R}^N)}^N\leq  \left(\tau\underset{\mathbb{R}^N}{\int}|w_n|^{N}w_{L,n}^{N(\gamma-1)}\dx+\underset{\mathbb{R}^N}{\int}C(\tau)|w_n|^{\nu-1}w_nw_{L,n}^{N(\gamma-1)}\Phi_{\alpha_0,N-2}(w_n)\dx\right) . 
      \end{equation}
      From equation \eqref{4.36}, we obtain
      \begin{equation}
        \|\Xi(w_n)\|_{W_{V_0}^{1,N}(\mathbb{R}^N)}^N\leq \tau\gamma^NC_*^N\|\Xi(w_n)\|_{W_{V_0}^{1,N}(\mathbb{R}^N)}^N+\underset{\mathbb{R}^N}{\int}C(\tau)|w_n|^{\nu-1}w_nw_{L,n}^{N(\gamma-1)}\Phi_{\alpha_0,N-2}(w_n)\dx.   
      \end{equation}
      Choose $\tau>0$ such that $0<\tau<\gamma^{-N}C_*^{-N}$ and using generalized H\"older's inequality, boundedness of $w_n$, Lemmas \ref{lemma3.4} and \ref{lemma2.2}, one leads to 
      \begin{equation}
          \|w_nw_{n,L}^{\gamma-1}\|_{L^s(\mathbb{R}^N)}^N\leq\tilde{C}\gamma^N\|w_nw_{n,L}^{\gamma-1}\|_{L^{N\mu}(\mathbb{R}^N)}^{N\mu}
      \end{equation}
      for some $\mu>1.$ Since $w_{n,L}\leq w_n$ and as $L\rightarrow+\infty$, Fatou's lemma implies that 
      \begin{equation}\label{4.40}
         \|w_n\|_{L^{\gamma s}(\mathbb{R}^N)}\leq(\tilde{C})^{\frac{1}{{N\gamma}}}\gamma^{\frac{1}{\gamma}}\|w_n\|_{L^{N\mu\gamma}(\mathbb{R}^N)}^{\mu}.
      \end{equation}
      Choose $\gamma=\frac{s}{N\mu}$, then $\gamma^2N\mu=\gamma s$. Replace $\gamma$ by $\gamma^2$ and using equation\eqref{4.40} we obtained that
      \begin{equation*}
         \|w_n\|_{L^{\gamma^2 s}(\mathbb{R}^N)}\leq(\tilde{C})^{\frac{1}{{N\gamma^2}}}\gamma^{\frac{1}{\gamma^2}}\|w_n\|_{L^{\gamma s}(\mathbb{R}^N)}^{\mu}\leq(\tilde{C})^{\frac{1}{N}(\frac{1}{\gamma}+\frac{1}{\gamma^2})}\gamma^{(\frac{1}{\gamma}+\frac{1}{\gamma^2})}\|w_n\|_{L^{N\mu\gamma}(\mathbb{R}^N)}^{\mu} . 
      \end{equation*}
      On repeating the same arguments $k$ times, we obtain
     \begin{equation*}
        \|w_n\|_{L^{\gamma^k s}(\mathbb{R}^N)}\leq(\tilde{C})^{\frac{1}{N}\sum_{j=1}^k(\frac{1}{\gamma^j})}\gamma^{\sum_{j=1}^k(\frac{1}{\gamma^j})}\|w_n\|_{L^{N\mu\gamma}(\mathbb{R}^N)}^{\mu}   .
     \end{equation*}
     As $k\rightarrow\infty$ in above equation, for any $n\in\mathbb{N}$, we get $ \|w_n\|_{L^\infty(\mathbb{R}^N)}\leq K.$
     for some constant $K>0.$
     Now we will show that $$\underset{|x|\rightarrow+\infty}{{\lim}}w_n(x)=0,~\text{uniformly in } n\in\mathbb{N.}$$
     Due to embedding $W_{V_{\epsilon}}^{1,p}(\mathbb{R}^N)\cap W_{V_{\epsilon}}^{1,N}(\mathbb{R}^N)\hookrightarrow W_{V_{0}}^{1,p}(\mathbb{R}^N)\cap W_{V_{0}}^{1,N}(\mathbb{R}^N)$, we deduce that $\{w_n\}_{n\geq1}$ is bounded in $\boldsymbol{Y}$. Moreover,
     $w_n\rightarrow w~\text{in~}~L^{t}(\mathbb{R}^N),~\forall~t\in[p,p*]\cup [N,+\infty) $ and $w_n\rightarrow w$ a.e. in $\mathbb{R}^N.$ It can be observed that $w_n$ solves 
     \begin{equation}\label{4.26}
         -\Delta_pw_n-\Delta_Nw_n+V_0(w_n^{p-1}+w_n^{N-1})\leq\tau w_n^{N-1}+C(\tau)w_n^{\nu-1}\Phi_{\alpha_0,N-2}(w_n)
     \end{equation}
     in the weak sense. Using the arguments used in \cite{Liu-1969} it can be proved that $$\tau w_n^{N-1}+C(\tau)w_n^{\nu-1}\Phi_{\alpha_0,N-2}(w_n)\in L^{\frac{N}{N-1}}(\mathbb{R}^N)\subset Y^* $$ where $\boldsymbol{Y}^*$ denotes the dual space of $\boldsymbol{Y}$. Define the operator $T: \boldsymbol{Y}\rightarrow\boldsymbol{Y}^*$
     \begin{equation}
         \langle T(u),v\rangle=\sum_{s\in\{p,N\}}\underset{\mathbb{R}^N}{\int}|\nabla u|^{s-2}\nabla u\cdot\nabla v+ V_0|u|^{s-2}uv\dx
         \end{equation}
         for all $u,v\in\boldsymbol{Y}$. By following the arguments used in \cite{D.K.-2025}, we can show that $T$ is surjective. Due to which we have  $z_n\in\boldsymbol{Y}$ such that $T(z_n)=\tau w_n^{N-1}+C(\tau)w_n^{\nu-1}\Phi_{\alpha_0,N-2}(w_n)$. This means for all $v\in\boldsymbol{Y}$
         \begin{equation}\label{4.28}
            \sum_{s\in\{p,N\}}\underset{\mathbb{R}^N}{\int}|\nabla z_n|^{s-2}\nabla z_n\cdot\nabla v+ V_0|z_n|^{s-2}z_nv\dx=\underset{\mathbb{R}^N}{\int}\left(\tau w_n^{N-1}+C(\tau)w_n^{\nu-1}\Phi_{\alpha_0,N-2}(w_n)\right)v\dx.
         \end{equation}
    Assume $v=z_n^{-}$ as a test function in equation \eqref{4.28}, then we obtain $z_n^{-}=0$ a.e. in $\mathbb{R}^N$. Hence $z_n\geq0$ a.e. in $\mathbb{R}^N.$ On using equations \eqref{4.26}, \eqref{4.28} and \cite[Theorem 4.1]{Brasco-2022}, we conclude that 
    \begin{equation}\label{4.29}
        0\leq w_n\leq z_n ~\text{a.e. in}~ \mathbb{R}^N.
    \end{equation}
    By using the Young's inequality (i.e. for any $\vartheta>0$ and $a,b\geq 0$ with $\frac{1}{s}+\frac{1}{s'}=1$,
    $$ab\leq \vartheta a^s+C(\vartheta)b^{s'}$$
    Where $C$ is a constant depend on $\vartheta$), H\"older's inequality and Lemma \ref{lemma2.2} we can obtained the boundedness of $\{z_n\}_{n\geq1}$ in $\boldsymbol{Y}$. Consequently, up to a subsequence still denoted by itself 
    \begin{equation*}
        z_n\rightharpoonup z ~\text{in}~\boldsymbol{Y}~ \text{and}~z_n\rightarrow z~\text{a.e. in}~\mathbb{R}^N.
    \end{equation*}
    So,
    \begin{equation}\label{4.30}
      \|z\|_{W_{V_{0}}^{1,p}(\mathbb{R}^N)}^p+\|z\|_{W_{V_{0}}^{1,N}(\mathbb{R}^N)}^N=\underset{\mathbb{R}^N}{\int}\left(\tau w^{N-1}+C(\tau)w^{\nu-1}\Phi_{\alpha_0,N-2}(w)\right)z\dx.  
    \end{equation}
    Next we will show that $\underset{n\rightarrow+\infty}{\text{lim}}\underset{\mathbb{R}^N}{\int} w_n^{N-1}z_n\dx=\underset{\mathbb{R}^N}{\int} w^{N-1}z\dx$.\\
    Since $w^{N-1}\in L^{\frac{N}{N-1}}(\mathbb{R}^N)$ and using the weak convergence of $z_n$ in $L^{N}(\mathbb{R}^N)$, we have
    \begin{equation}\label{4.31 b}
    \underset{n\rightarrow+\infty}{\text{lim}}\underset{\mathbb{R}^N}{\int} w^{N-1}(z_n-z)\dx=0.
    \end{equation}
    On the other side, due to $w_n\rightarrow w$ in $L^N(\mathbb{R}^N)$ and Lebesgue dominated convergence theorem, we have 
    $$\underset{n\rightarrow+\infty}{\text{lim}}\underset{\mathbb{R}^N}{\int}| w_n^{N-1}-w^{N-1}|^\frac{N}{N-1}\dx=0.$$
    Now
    
    \begin{equation*}
    \begin{split}
      \underset{\mathbb{R}^N}{\int} (w_n^{N-1}z_n-w^{N-1} z)\dx=\underset{\mathbb{R}^N}{\int} (w_n^{N-1}z_n-w^{N-1}z_n+w^{N-1}z_n-w^{N-1} z)\dx\\
      =\underset{\mathbb{R}^N}{\int} (w_n^{N-1}z_n-w^{N-1}z_n)\dx+\underset{\mathbb{R}^N}{\int}(w^{N-1}z_n-w^{N-1} z)\dx .
      \end{split}
    \end{equation*}
    Using H\"older's inequality and boundedness of $\{z_n\}$ in $\boldsymbol{Y}$, we have 
    \begin{equation}\label{4.32-1}
        \underset{n\rightarrow+\infty}{\text{lim}}\underset{\mathbb{R}^N}{\int}(w_n^{N-1}z_n-w^{N-1}z_n)\dx=0. 
    \end{equation}
    From equations \eqref{4.31 b} and \eqref{4.32-1} it follows that
    \begin{equation}\label{4.33a}
    \underset{n\rightarrow+\infty}{\text{lim}}\underset{\mathbb{R}^N}{\int} w_n^{N-1}z_ndx=\underset{\mathbb{R}^N}{\int} w^{N-1}z\dx.
    \end{equation}
     Similarly, it follows that
     \begin{equation}\label{4.33b}
      \underset{n\rightarrow+\infty}{\text{lim}}\underset{\mathbb{R}^N}{\int} w_n^{N-1}z_n\Phi _{\alpha_0,N-2}(w_n)\dx=\underset{\mathbb{R}^N}{\int} w^{N-1}z\Phi _{\alpha_0,N-2}(w)\dx.
      \end{equation}
      By virtue of equations \eqref{4.33a} and \eqref{4.33b}
      $$ \|z_n\|_{W_{V_0}^{1,p}(\mathbb{R}^N)}^p+\|z_n\|_{W_{V_0}^{1,N}(\mathbb{R}^N)}^N=\|z\|_{W_{V_0}^{1,p}(\mathbb{R}^N)}^p+\|z\|_{W_{V_0}^{1,N}(\mathbb{R}^N)}^N+o_n(1)$$
      as $n\rightarrow+\infty.$
      From Br\'ezis-Lieb lemma \cite{Brezis-1983}, it follows that $z_n\rightarrow z$ in $\boldsymbol{Y}$. Since $0\leq w_n\leq z_n$ a.e. in $\mathbb{R}^N$, again using Moser iteration argument as in the first part of this lemma, we have $$\|z_n\|_{L^{\infty}(\mathbb{R}^N)}\leq C$$ for some constant $C>0.$
      Hence from \cite{Tong-2025} and \eqref{4.29}, it follows that $$\underset{|x|\rightarrow+\infty}{{\lim}}w_n(x)=0,~\text{uniformly in } n\in\mathbb{N.}$$ This completes the proof. 
      \end{proof}
      \section{Proof of Main Theorem}\label{Section5} Now we are empowered enough to conclude the proof of the main theorem of this article.
  \begin{proof}[Proof of Theorem \ref{main theorem 1}] 
  We will prove that there exists $\tilde{\epsilon}$, $\beta>0$ such that for every $\epsilon\in(0,\tilde{\epsilon})$, $\beta\in(0,\hat{\beta})$ and every solution $u_{\epsilon,\beta}$ of problem \eqref{Auxiliary problem} satisfies
  \begin{equation}\label{4.42}
      \|u_{\epsilon,\beta}\|_{L^{\infty}(\Lambda_{\epsilon}^c)}<a.
  \end{equation}
   On contrary suppose that \eqref{4.42} does not hold, i.e., for some subsequence $\epsilon_n, \beta_n\rightarrow0$ we have $\{u_n\}(:=\{u_{\epsilon_n,\beta_n}\})$  such that $I_{\epsilon_n,\beta_n}(u_n)=c_{\epsilon_n,\beta_n}$, $0\in\partial I_{\epsilon_n,\beta_n}(u_n)$ and 
   \begin{equation}\label{4.43}
       \|u_{\epsilon,\beta}\|_{L^{\infty}(\Lambda_{\epsilon}^c)}\geq a.
   \end{equation}
   From Lemma \ref{lemma4.3}, we have $(y_n)\subset\mathbb{R}^N$ such that $w_n(x)=u_n(x+y_n)\rightarrow w$ in $\boldsymbol{Y}$ and $\epsilon_ny_n\rightarrow y_0$ with $V(y_0)=V_0.$ Set $r>0$ such that $B_r(y_0)\subset B_{2r}(y_0)\subset\Lambda$ and hence 
   $$ B_{\frac{r}{\epsilon_n}}\left(\frac{y
   _0}{\epsilon_n}\right)\subset \Lambda_{\epsilon_n}.$$
   So, for all $y\in B_{\frac{r}{\epsilon_n}}\left(y_n\right),$
   $$\left|y-\frac{y_0}{\epsilon_n}\right|\leq \left|y-y_n\right|+\left|y_n-\frac{y_0}{\epsilon_n}\right|<\frac{2r}{\epsilon_n}$$
   for large $n$ and in that case, there holds
   \begin{equation}\label{4.44}
       \Lambda_{\epsilon_n}^c\subset B_{\frac{r}{\epsilon_n}}^c(y_n).
   \end{equation}
   By Lemma \ref{lemma4.4}, we have 
   $$w_n(x)\rightarrow0~\text{as}~|x|\rightarrow+\infty~\text{uniformly   in }  n.$$
  Hence there is $\mathcal{R}>0$ such that $$|w_n(x)|<a,~\forall~|x|\geq\mathcal{R},~n\in \mathbb{N.}$$
  So $u_n(x)<a$ for $x\in B_{\mathcal{R}}^c(y_n)$ and from \eqref{4.44} there exist $n_0\in\mathbb{N}$ such that 
  $$\Lambda_{\epsilon_n}^c\subset B_{\frac{r}{\epsilon_n}}^c(y_n)\subset B_{\mathcal{R}}^c(y_n),~\forall~n\geq n_0. $$
  This means $u_n(x)<a$ for all $x\in\Lambda_{\epsilon_n}^c$ and $n\geq n_0$, which contradicts \eqref{4.43}. Hence, our claim holds. Since $u_{\epsilon,\beta}$ is solution for problem \eqref{Auxiliary problem} and satisfies
  \begin{equation}\label{4.45}
      \|u_{\epsilon,\beta}\|_{L^{\infty}(\Lambda_{\epsilon}^c)}<a.
  \end{equation}
   From Remark \ref{remark1}, it follows that $u_{\epsilon,\beta}$ is also a solution for main problem \eqref{main problem} for all $\epsilon\in(0,\tilde{\epsilon})$, and  $\beta\in(0,\hat{\beta}).$ Next, we will explore the behaviour of the maximum points of $u_{\epsilon,\beta}$ for small enough $\epsilon$ and $\beta.$ From assumption \ref{g3}, one can conclude that there is $\tau\in(0,a)$ such that
  \begin{equation}\label{4.46}
  |g(\epsilon x,s)s|\leq\frac{V_0}{2}s^N,~\forall~x\in\mathbb{R}^N~\text{and}~s\leq\tau.
  \end{equation}
  From \eqref{4.45} we have 
  $$ \|u_n\|_{L^{\infty}(B_{\mathcal{R}}^c(y_n))}<\tau. $$
  Also, we can assume that \begin{equation}\label{4.47} \|u_n\|_{L^{\infty}(B_{\mathcal{R}}(y_n))}\geq \tau.
  \end{equation}
  Suppose \eqref{4.47} does not hold. Then, using the fact that $0\in\partial I_{\epsilon_n,\beta_n}(u_n)$ and from \eqref{4.46}, we have 
  \begin{equation*}
      0<\underset{\mathbb{R}^N}{\int}\left(|\nabla u_n|^p+V_0|u_n|^p\right)\dx+\underset{\mathbb{R}^N}{\int} \left(|\nabla u_n|^N+V_0|u_n|^N\right)\dx\leq\frac{V_0}{2}\underset{\mathbb{R}^N}{\int}|u_n|^N\dx
  \end{equation*}
  which leads to $u_n\rightarrow0$ as $n\rightarrow\infty.$ Consequently $I_{\epsilon_n,\beta_n}(u_n)\rightarrow0$; which is not possible. Hence \eqref{4.47} holds. Let $x_n$ be the global maximum point of $u_n$. This means $x_n\in B_{\mathcal{R}}(y_n).$ Set $x_n=y_n+z_n$ for some $z_n\in B_{\mathcal{R}}.$ So $\epsilon_nx_n=\epsilon_ny_n+\epsilon_nz_n\rightarrow y_0$ as $n\rightarrow+\infty.$ Now, using the continuity of the function $V$, we have
      $\underset{n\rightarrow+\infty}{\text{lim}}V(\epsilon_nx_n)=V_0$,
  which concludes the proof.
  \end{proof}
\section*{Contributions}
All authors contributed to the conception and design of the study. All authors contributed to the preparation of the material, data collection, and analysis. The authors read and approved the final manuscript.
\section*{Conflict of Interests} The authors declare that they have no conflict of interest.

\end{document}